\documentclass[11pt]{amsart}
\usepackage[all]{xy}

\usepackage{amsmath}
\usepackage{amsfonts}
\usepackage{amssymb}
\usepackage{mathrsfs}

\DeclareFontEncoding{OT2}{}{} 
\newcommand{\textcyr}[1]{%
 {\fontencoding{OT2}\fontfamily{wncyr}\fontseries{m}\fontshape{n}
 \selectfont #1}}
\newcommand{\Sha}{{\!\be\lbe\mbox{\textcyr{Sh}}}}

\numberwithin{equation}{section}

\def\Gtil{{\widetilde{G}}}
\def\bh{{\mathbb H}}

\def\spec{{\rm{Spec}}\,}

\def\img{{\rm{Im}}\,}

\def\der{\e\rm{der}}

\def\be{\kern -.1em}
\def\lbe{\kern -.025em}

\def\krn{{\rm{Ker}}\e }
\def\img{{\rm{Im}}\e }
\def\cok{{\rm{Coker}}}

\def\tor{\e\rm{tor}}

\def\ra{\rightarrow}
\def\lra{\longrightarrow}
\def\e{\kern 0.08em}
\def\le{\kern 0.03em}
\def\ng{\kern -0.04em}

\def\krn{{\rm{Ker}}\,}
\def\cok{{\rm{Coker}}\,}

\newtheorem{lemma}{Lemma}[section]
\newtheorem{theorem}[lemma]{Theorem}
\newtheorem{corollary}[lemma]{Corollary}
\newtheorem{proposition}[lemma]{Proposition}
\theoremstyle{definition}
\newtheorem{definition}[lemma]{Definition}
\theoremstyle{remark}
\newtheorem{remark}[lemma]{Remark}
\newtheorem{remarks}[lemma]{Remarks}
\newtheorem{example}[lemma]{Example}
\newtheorem{examples}[lemma]{Examples}

\begin{document}

\title[Crossed modules and nonabelian cohomology ]{Quasi-abelian crossed
modules and nonabelian cohomology}


\author{Cristian D. Gonz\'alez-Avil\'es}
\address{Departamento de Matem\'aticas, Universidad de La Serena, Chile}
\email{cgonzalez@userena.cl}


\thanks{The author is partially supported by Fondecyt grant
1080025}

\dedicatory{\`A Jean-Claude Douai, \`a l'occasion de son depart en
retraite.}

\maketitle

\begin{abstract} We extend the work of M.Borovoi on
the nonabelian Galois cohomology of linear reductive algebraic
groups over number fields to a general base scheme. As an
application, we obtain new results on the arithmetic of such groups
over global function fields.
\end{abstract}

\section{Introduction}

Let $G$ be a linear reductive algebraic group over a number field
$K$. In \cite{Bor,Bor2}, M.Borovoi studied the Galois cohomology
sets $H^{\e i}(K,G)$ ($i=1,2$) by relating them to certain ``abelian
cohomology groups" $H^{\e i}_{\rm{ab}}(K,G)$. In particular, it was
shown in \cite{Bor} that there exists a {\it surjective}
abelianization map ${\rm{ab}}^{1}\colon H^{\le 1} (K,G)\ra H^{\e
1}_{\rm{ab}}(K,G)$. Later, in \cite{BKG}, Proposition 6.6, a new
proof of this fact was given which suggested to the author the
existence of a prolongation, at least for certain types of crossed
modules, of L.Breen's fundamental exact sequence \cite{Br}, (4.2.2)
(see Proposition 2.4). It was expected that such an extension of
Breen's sequence would provide, via work of J.-C.Douai, a common
explanation for the validity of both \cite{BKG}, Proposition 6.5,
and \cite{Bor2}, Theorem 5.5. The purpose of this paper is to
establish such a prolongation of Breen's sequence over any base
scheme $S$, thus confirming the above expectation (see Theorem 1.1
below). In particular, we extend Borovoi's abelian cohomology theory
to this general setting. To be precise, let $S$ be any scheme, let
$S_{\e\rm{fl}}$ be the small fppf site over $S$ and let $G$ be a
reductive group scheme over $S$. Write $G^{\der}$ for the derived
group of $G$, $\widetilde{G}$ for the simply-connected central cover
of $G^{\der}$ and $\mu$ for the fundamental group of $G^{\der}$.
Then $G$ defines a crossed module $(\widetilde{G}\be\ra\be G)$ on
$S_{\e\rm{fl}}$ which is quasi-isomorphic to the abelian crossed
module $\big(Z\lbe\big(\e\widetilde{G}\,\big)\!\ra\be Z(G)\big)$
(thus it is {\it quasi-abelian}, in the sense of Definition 3.2).
The abelian (flat) cohomology groups $H^{\e
i}_{\rm{ab}}(S_{\e\rm{fl}} ,G)$ of $G$ are by definition the fppf
hypercohomology groups ${\bh}^{\e i}\be\big(S_{\e\rm{fl}},Z\lbe\big(
\e\widetilde{G}\,\big)\!\ra\be Z(G)\big)$. Now let $H^{\le
2}\lbe\big (S_{\e\rm{fl}},\widetilde{G}\e\big)$ be the second
cohomology set introduced by J.Giraud in \cite{Gi}, Chapter IV. It
contains a distinguished element $\varepsilon_{\lbe\widetilde{G}}$,
the so-called {\it unit class}, and a subset
$H^{2}\lbe\big(S_{\e\rm{fl}},\widetilde{G}\e\big) ^{\lbe\prime}$ of
{\it neutral classes} containing $\varepsilon_{\lbe\widetilde{G}}$.
We regard $H^{2}\lbe\big(S_{\e\rm{fl}},\widetilde{G}\e\big)$ as a
pointed set with basepoint $\varepsilon_{\widetilde{G}}$. Analogous
definitions apply to $H^{2}(S_{\e\rm{fl}},G)$. Then, by using ideas
of L.Breen \cite{Br}, J.Giraud \cite{Gi} and M.Borovoi \cite{Bor2},
we are able to define abelianization maps ${\rm{ab}}^{i}\colon
H^{\le i} (S_{\e\rm{fl}},G)\ra H^{\e i}_{\rm{ab}}(S_{\e\rm{fl}},G)$
for $i=0,1,2$ (which generalize those defined in \cite{Bor,Bor2}) so
that the following theorem holds. To simplify the notation, let
$H^{i}(G)$ and $H^{i}_{\rm{ab}}(G)$ denote $H^{i}(S_{\e\rm{fl}},G)$
and $H^{i}_{\rm{ab}}(S_{\e\rm{fl}},G)$, respectively.

\begin{theorem} Let $G$ be a reductive group scheme over a scheme $S$.
Then there exists a sequence of flat (fppf) cohomology sets
$$
\begin{array}{rcl}
1&\longrightarrow&\mu(S)\ra\widetilde{G}(S)\ra G (S)
\overset{\rm{ab}^{0}}\longrightarrow H^{\e 0}_{\rm{ab}}(G \e)\ra H^{\le
1}\be\big(\widetilde{G }\e\big)
\ra H^{\le
1}(G \e)\\
&\overset{\rm{ab}^{1}}\longrightarrow& H^{\e 1}_{\lbe\rm{ab}}(G
\e)\overset{\delta_{1}}\longrightarrow
H^{2}\lbe\big(\widetilde{G}\e\big) \ra H^{\e 2}(G \e)
\overset{{\rm{ab}}^{\lbe 2}} \longrightarrow H^{\e 2}_{\rm{ab}}(G
\e)\ra H^{\e 3}\big(Z\big(
\widetilde{G}\e\big)\big)\\
&\longrightarrow& H^{\e 3}(Z(G\e))\ra\dots,
\end{array}
$$
which is an exact sequence of pointed sets at every term except
$H^{\e 1}_{{\rm{ab}}} (G)$, where a class $y\in H^{\e
1}_{{\rm{ab}}}(G)$ is in the image of $\,{\rm{ab}}^{1}$ if, and only
if, $\delta_{1}(y)\in H^{2}\lbe\big(\widetilde{G}\e\big)
^{\lbe\prime}$.
\end{theorem}

The above theorem is, in fact, a corollary of Theorem 4.2, which
applies to any quasi-abelian crossed module. When $S$ is the
spectrum of a number field $K$, the theorem shows, as expected, that
one and the same fact underlie the validity of both \cite{BKG},
Proposition 6.6, and \cite{Bor2}, Theorem 5.5. Namely, that all
classes in $H^{2}\lbe\big(K_{\text{fl}},\widetilde{G}\e\big)$ are
neutral. See Section 5.

The theorem yields the following {\it integral} version of
\cite{Bor}, Theorem 5.7.

\begin{corollary} Let $G$ be a reductive group scheme over the
spectrum $S$ of the ring of integers of a number field. Then the
first abelianization map ${\rm{ab}}^{1}\colon H^{1}(S_{\e\rm{fl}},
G) \ra H^{1}_{{\rm{ab}}}(S_{\e\rm{fl}},G)$ is surjective.
\end{corollary}

The following is a brief summary of the paper. In Section 1 we
review basic facts from Breen's nonabelian cohomology theory of
crossed modules. In Section 2 we introduce quasi-abelian crossed
modules and establish a part of the sequence appearing in Theorem
4.2 (which is the main theorem of the paper). The basic reference
for this Section is the Giraud-Grothendieck nonabelian cohomology
theory \cite{Gi}. In Section 4, following M.Borovoi, we define the
map ${\rm{ab}}^{2}$ and obtain the latter part of the sequence
appearing in Theorem 4.2. In Section 5, which concludes the paper,
we discuss applications of Theorem 1.1 to linear reductive algebraic
groups over certain types of fields, especially of positive
characteristic. Additional applications are discussed in \cite{GA}.

\section*{Acknowledgements}
I am very grateful to Mikhail Borovoi and Lawrence Breen for their
enlightening comments. I am also very grateful to Cyril Demarche,
who read a preliminary version of this paper and suggested proofs of
Proposition 3.11, Corollary 3.12 and Lemma 3.16. Finally, I thank
Jean-Claude Douai for sending me copies of some of his papers
(including his thesis), and Rick Jardine and Dino Lorenzini for
additional bibliographical assistance.

\section{Preliminaries}

Let $E$ be a site and let $\widetilde{E}$ be the topos defined by
$E$. If $G$ is a group of $\widetilde{E}$, $\text{Inn}(G)$ will
denote the quotient $G/Z(G)$, which is canonically isomorphic to the
group of inner automorphisms of $G$.

We begin by recalling the basic properties of the nonabelian
cohomology theory of crossed modules developed in
\cite{Br}\footnote{In contrast to \cite{Br}, we work with left
crossed modules and right torsors throughout. See \cite{Br}, p.416,
for the equivalence of both approaches. Thus references to results
from \cite{Br} below are, in fact, to their opposite-hand
versions.}.

\begin{definition} A (left) crossed module on $E$
consists of a homomorphism $\partial\e\colon F\ra G$ of groups of
$\widetilde{E}$ together with a left action of $G$ on $ F$, denoted
by $(g,f)\mapsto\!\phantom{.}^{g}\!\lbe f$, such that
$$
\partial(\!\phantom{.}^{g}\!\lbe f)=g\partial(f)\e g^{-1}
$$
and
$$
\!\phantom{.}^{\partial(\lbe f\lbe)}\! f^{\e\prime}=ff^{\e\prime}\e
f^{-1},
$$
for every $g\in G$ and $f, f^{\e\prime}\in F$.
\end{definition}

The definition immediately implies that $\krn\partial$ is central in
$ F$ and $\img\partial$ is normal in $G$. Further, $\partial$ is a
$G$-homomorphism for the given action of $G$ on $F$ and the left
action of $G$ on itself via inner automorphisms.

Crossed modules will be regarded as complexes of length two $\big(
F\!\overset{\partial}\ra\be G\e\big)$, with $F$ and $G$ placed in
degrees $-1$ and $0$, respectively.

\noindent\begin{examples}\indent
\begin{enumerate}
\item[(i)] If $\partial\colon\!  F\ra G$ is a
homomorphism of abelian groups of $\widetilde{E}$ and $ G$ acts
trivially on $F$, then $\big( F\!\overset{\partial}\ra\be G\e\big)$
is a crossed module. Such crossed modules are called {\it abelian}.
\item[(ii)] If $ G$ is a group of $\widetilde{E}$, $F$ is a
normal subgroup of $G$ and $G$ acts on
$F$ by conjugation, then $(F\hookrightarrow G)$ is a crossed module.
\item[(iii)] Let $S$ be a scheme and let $E$ be a standard
site on $S$, i.e., every representable presheaf on $E$ is a sheaf
and finite fibered products exist in $E$. Examples include the small
fppf and \'etale sites over $S$ \cite{SGA3}, Proposition
IV.6.3.1(iii). Let $G$ be a reductive $S$-group scheme with derived
group $G^{\der}$, let $\widetilde{G}$ be the simply-connected
central cover of $G^{\der}$ and let $\partial\e\colon\widetilde{
G}\ra G$ be the composition $\widetilde{ G}\twoheadrightarrow
G^{\der}\hookrightarrow G$. Then there exists a canonical
``conjugation" action of $G$ on $\widetilde{G}$ so that the complex
$\big(\widetilde{ G}\overset{\partial}\ra G\e\big)$ of representable
sheaves on $E$ is a crossed module on $E$. See \cite{Br2}, Example
1.9, p.28. Further, if $Z(G)$ denotes the center of $G$, then the
induced action of $Z(G)$ on $\widetilde{G}$ is trivial.
\end{enumerate}
\end{examples}

Let $(F\!\overset{\partial}\ra\be G)$ be a crossed module. For
$i=-1,0,1$, let $H^{\le i}(E, F\!\ra\be G)$ be the sets $H^{\le
i}(\widetilde{E}, F\!\ra\be \check{G}\e)$ defined in \cite{Br},
p.426, where $\check{G}$ is the symmetric of $G$ \cite{Gi},
Definition III.1.1.3, p.106. If $( F\ra G)$ is an {\it abelian}
crossed module (see Example 2.2(i)), the pointed sets $H^{\e i}(E,
F\!\ra\be G)$ coincide with the usual flat hypercohomology groups of
the complex of abelian groups $( F\ra G)$. Further, if $G$ is a
group of $\widetilde{E}$, then the pointed sets $H^{\e i}(E,1\ra
G)$, where $i=0,1$, agree with the usual cohomology sets $H^{\le
i}(E, G)$. See \cite{Br}, p.427, line 3. In particular, $H^{\e
1}(E,1\ra G)=H^{\le 1}(E, G)$ is the set of isomorphism classes of
right $G$-torsors on $E$. We have [op.cit.], (4.2.1),
\begin{equation}\label{-1}
H^{\e -1}\lbe\big(E, F\!\be\overset{\partial}\ra\be \be G\big)=H^{\e
0}(E,\krn\partial\e),
\end{equation}
which is an abelian group. Further, $H^{\e 0}(E, F\!\ra\be G)$ is
the set of isomorphism classes of ``$(F,G)$-torsors", i.e, pairs
$(Q,t)$ where $Q$ is an $F$-bitorsor and $t\colon Q\ra G$ is an
$F$-equivariant map for the right action of $F$ on $G$ via
$\partial$. This set is naturally equipped with a group structure
given by the contracted product of $F$-bitorsors. See \cite{Br},
p.432. In order to describe $H^{1}\be\big(E, F\!\be\ra\be \be
G\big)$, we first note that the crossed module $(F\!\be\ra\be G)$ is
functorially associated to the opposite $\mathcal C$ of the gr-stack
of $(F,G\e)$-torsors on $E$. See [op.cit.], Theorem 4.6, p.433.
Then, by [op.cit.], Theorem 6.2, p.440, $H^{1}\be\big(E,
F\!\be\ra\be \be G\big)$ is in natural bijection with the set
$H^{1}(\mathcal C\e)$ of equivalence classes of (right) $\mathcal
C$-torsors on $E$. The sets thus defined behave functorially with
respect to inverse images, i.e., if $u\colon E^{\e\prime}\ra E$ is a
morphism of sites and $i=0,1$, then there exist maps
\begin{equation}\label{pbck}
h^{i}(u,F\!\ra\be G)\colon H^{\e i}(E,F\!\ra\be G)\ra H^{\e
i}(E^{\e\prime},u^{\lbe *}F\!\ra\be u^{\lbe *}G\e)
\end{equation}
which coincide with those defined in \cite{Gi}, Proposition V.1.5.1,
p.316, when $(F\!\ra\be G)=(1\ra G)$. See \cite{Br}, \S6. For ease
of notation, we will sometimes write $H^{\le i}(G)$ and $H^{\e
i}(F\!\ra\be G)$ for $H^{\le i}(E, G)$ and $H^{\e i}(E, F\!\ra\be
G)$, respectively.

A morphism of crossed modules
$(F_{1}\!\overset{\partial_{1}}\longrightarrow\be
 G_{1})\ra( F_{2}\!\overset{\partial_{2}}\longrightarrow
\be G_{2})$ (see [op.cit.], p.416, for the definition) is called a
{\it quasi-isomorphism} if the induced group homomorphisms
$\krn\partial_{\e 1}\ra\krn
\partial_{\e 2}$ and
$\cok\partial_{\e 1}\ra\cok\partial_{\e 2}$ are isomorphisms. Since
quasi-isomorphic crossed modules define equivalent gr-stacks (cf.
\cite{SGA4}, XVIII, 1.4.12), a quasi-isomorphism of crossed modules
$( F_{1}\!\ra\be  G_{1})\ra( F_{2}\!\ra \be G_{2})$ induces
bijections
$$
H^{\e i}\lbe\big(E, F_{1}\!\be\ra\be\be
 G_{1}\le\big)\overset{\!\sim}\ra H^{\e
i}\lbe\big(E, F_{2}\!\be\ra\be\be  G_{2}\le\big),
$$
for $i=-1,0,1$.

\begin{examples}\noindent
\begin{enumerate}
\item[(i)] Let $\big( F\!\overset{\partial}\ra\be G\e\big)$
be a crossed module on $E$ such that $\partial$ is {\it injective}.
Then $\big( F\!\overset{\partial}\ra\be G\e\big)$ is quasi-isomorphic to
$(1\ra\cok\partial\e)$ and there exist canonical bijections
$$
H^{\e i}\lbe\big(E, F\!\overset{\partial}\ra\be
 G\e\big)\simeq H^{\e i}(E,\cok\partial\e)
$$
for $i=0,1$.

\item[(ii)] Let $\big( F\!\overset{\partial}\ra\be G\e\big)$
be a crossed module on $E$ such that $\partial$ is {\it surjective}.
Then there exists a quasi-isomorphism $(\krn\partial\!\ra\be 1)\ra
\big( F\!\overset{\partial}\ra\be G\e\big)$ which induces bijections
$$
H^{\e i}\lbe\big(E, F\!\overset{\partial}\ra\be
 G\e\big)\overset{\!\sim}\ra H^{\e
i+1}(E,\krn\partial\e),
$$
where $i=-1,0,1$. Note that $H^{\e 2}(E,\krn\partial\e)$ is the
usual second cohomology group of the abelian group $\krn\partial$ of
$\widetilde{E}$.
\end{enumerate}
\end{examples}

Let $(F\!\overset{\partial}\ra\be G)$ be a crossed module as above.
Then $\partial$ induces a map $\partial^{\e(1)}\colon H^{1}(F)\ra
H^{1}(G)$ which maps the class of an $F$-torsor $Q$ to the class of
the $G$-torsor $Q\be\wedge^{\be F}\be G$ \cite{Gi}, Proposition
III.1.3.6, p.116.

\begin{proposition} Let $(F\!\overset{\partial}\ra\be G)$
be a crossed module of $\,E$. Then there exists an exact sequence of
pointed sets
$$\begin{array}{rcl}
1&\ra &H^{\e -1}(F\!\ra\be G)\ra H^{\le 0}( F)\ra H^{\le 0}(
G)\overset{\psi_{\le 0}}\longrightarrow
H^{\e 0}( F\!\ra\be G)\\
&\overset{\delta_{\le 0}^{\e\prime}}\longrightarrow & H^{\le
1}(F)\overset{\partial^{\le(1)}}\longrightarrow H^{\le 1}(
G)\overset{\psi_{\le 1}}\longrightarrow H^{\e 1}( F\!\ra\be G),
\end{array}
$$
where the maps $\psi_{\e i}$, $i=0,1$, are induced by the embedding
of crossed modules $(1\ra G)\hookrightarrow( F\!\ra\be G)$ and the
map $\delta_{\le 0}^{\e\prime}$ is defined below.
\end{proposition}
\begin{proof} See \cite{Br}, (4.2.2).
\end{proof}

\begin{remarks}\indent
\begin{enumerate}
\item[(a)] The map $\delta_{\le 0}^{\e\prime}$ is defined as
follows. If a class $c\in H^{\e 0}(F\!\ra\be G)$ is represented by
an $(F,G)$-torsor $(Q,t)$, where $Q$ is an $F$-bitorsor, then
$\delta_{\le 0}^{\e\prime}(c)\in H^{\le 1}(F)$ is represented by $Q$
regarded only as a right $F$-torsor. See \cite{Br}, p.414, line -10.
\item[(b)] The map $\psi_{0}$ (which is denoted $\alpha$ in
\cite{Br}, (2.16.1), p.414) is a homomorphism of groups. See
[op.cit.], p.432. Thus it defines a right action of $H^{0}(G)$ on
$H^{\e 0}( F\!\ra\be G)$ and it follows without difficulty from (a)
and Proposition 2.4 that $\delta_{\le 0}^{\e\prime}$ induces an
injection
$$
H^{\e 0}( F\!\ra\be G)/H^{0}(G)\ra H^{\le 1}(F).
$$
\item[(c)] The group $H^{\e 0}(F\!\ra\be G)$ acts on the right on the set
$H^{\le 1}(F)$ as follows. If $p\in H^{\le 1}(F)$ is represented by
an $F$-torsor $P$ and $c\in H^{\e 0}(F\!\ra\be G)$ is represented by
an $(F,G)$-torsor $(Q,t)$, then $p\cdot c\in H^{\le 1}(F)$ is
represented by $P\be\wedge^{\be F}\be Q$. This action is compatible
with the map $\delta_{\le 0}^{\e\prime}$, i.e., $\delta_{\le
0}^{\e\prime}(c_{1}c_{2})=\delta_{\le 0}^{\e\prime}(c_{1})\cdot
c_{2}$ for all $c_{1},c_{2}\in H^{\e 0}(F\!\ra\be G)$. In
particular, the action is transitive if, and only if, $\delta_{\le
0}^{\e\prime}$ is surjective. In addition, the given action is
compatible with inverse images, i.e., if $u\colon E^{\e\prime}\ra E$
is a morphism of sites and $h^{i}(u,F\!\ra\be G)$ are the maps
\eqref{pbck}, then
$$
h^{1}(u,F\le)(p\cdot c\le)=h^{1}(u,F\le)(p)\cdot h^{0}(u,F\!\ra\be
G\le)(c)
$$
in $H^{1}(E^{\e\prime},u^{\lbe *}F\e)$. This follows from the fact
that $u^{\lbe *}(P\be\wedge^{\be F}\! Q)\simeq u^{\lbe
*}P\be\wedge^{\be u^{\lbe *}\be F}\! u^{\lbe *}Q\,$ \cite{Gi},
p.316, line -4.
\end{enumerate}
\end{remarks}

\section{Quasi-abelian crossed modules}

If $A$ is a group of $\widetilde{E}$, $H^{\e 2}(E, A)$ will denote
the second cohomology set of $A$ defined in \cite{Gi}, Definition
IV.3.1.3, p.247\,\footnote{For an excellent summary of Giraud's
theory, see \cite{DD}, \S1.2, and \cite{Deb}, \S1.}. It contains a
distinguished element $\varepsilon_{\be A}$, namely the class of the
gerbe $\text{TORS}(A)$ of $A$-torsors on $E$, which is called the
{\it unit class}. This class is contained in a subset $H^{\e 2}(E,
A\e)^{\e\prime}\subset H^{\e 2}(E, A\e)$ of {\it neutral classes}
[op.cit.], Definition IV.3.1.1, p.247. For convenience, we will
sometimes write $H^{\e 2}( A\e)$ and $H^{\e 2}( A\e)^{\e\prime}$ for
$H^{\e 2}(E, A\e)$ and $H^{\e 2}(E, A\e)^{\e\prime}$, respectively.
Both $H^{\e 2}( A\e)$ and $H^{\e 2}(A\e)^{\e\prime}$ will be
regarded as pointed sets with basepoint $\varepsilon_{\be A}$.

Let $(F\!\overset{\partial}\ra\be G)$ be a crossed module on $E$
such that $G=\img\partial\cdot{\rm{Cent}}_{G} (\img\partial\e)$ and
the induced action of $Z(G)$ on $Z(F)$ is trivial. By restricting
$\partial$ to $Z(F)$, we obtain a map $\partial^{\,\prime}\colon Z(
F)\ra Z(\img\partial\e)$. On the other hand, the equality $
G=\img\partial\,{\rm{Cent}}_{ G}(\img\partial\e)$ implies that
$Z(\img\partial\e)=\img\partial\cap Z( G)\,$. Thus $\partial$
induces a map $\partial_{Z}\colon Z( F)\ra Z( G)$, namely the
composite $Z( F)\overset{\partial^{\e\prime}}\ra
Z(\img\partial\e)\hookrightarrow Z( G)$, and $(Z(
F)\!\overset{\,\partial_{\lbe Z}}\longrightarrow\be Z(G))$ is an
abelian crossed module (see Example 2.2(i)). Further, there exists
an embedding of crossed modules
\begin{equation}\label{emb}
(Z( F)\!\overset{\,\partial_{\lbe Z}}\longrightarrow\be Z(
G))\hookrightarrow ( F\!\overset{\partial}\longrightarrow\be G).
\end{equation}
Note that, since $\krn\partial\subset Z( F)$, we have
$\krn\partial_{Z}=\krn\partial^{\,\prime}=\krn\partial$.

\begin{definition} Let $(F\!\overset{\partial}\ra\be G)$ be a crossed module
on $E$ such that $ G=\img\partial\,{\rm{Cent}}_{ G}(\img\partial\e)$
and $Z(G)$ acts trivially on $Z(F)$. Let $i\geq -1$ be an integer.
The $i$-th \textit{abelian cohomology group of $(F\!\ra\be G)$} is
by definition the hypercohomology group
$$
H^{\le i}_{\rm{ab}}(E, F\!\ra\be G)={\bh}^{\e i}(E,Z(
F)\!\overset{\,\partial_{Z}}\longrightarrow\be Z( G)),
$$
where $\partial_{Z}$ is as defined above.
\end{definition}

For $i=0,1$, the embedding of crossed modules \eqref{emb} defines
maps
\begin{equation}\label{bij}
\varphi_{i}\colon H^{\le i}_{\rm{ab}}(E, F\!\ra\be G)\ra H^{\le
i}(E, F\!\ra\be G).
\end{equation}
Further, the short exact sequence of complexes
\begin{equation}\label{zseq}
0\ra (0\ra
Z(G))\overset{j}\ra(Z(F)\overset{\partial_{Z}}\longrightarrow
Z(G))\overset{\pi}\ra(Z(F)\ra 0)\ra 0
\end{equation}
induces an exact sequence of abelian groups
\begin{equation}\label{long}
\dots\ra H^{\e i}(Z(G))\overset{j^{(i)}}\longrightarrow H^{\le
i}_{\rm{ab}}(F\!\ra\be G)\overset{\pi^{(i)}}\longrightarrow H^{\e
i+1}(Z(F))\overset{\partial_{\lbe Z}^{(i+1)}}\longrightarrow H^{\e
i+1}(Z(G))\ra\dots.
\end{equation}

\begin{definition} A crossed module $(F\!\overset{\partial}\ra\be G)$
on $E$ is called {\it quasi-abelian} if the following conditions
hold:
\begin{enumerate}
\item[(i)] the induced action of $Z(G)$ on $F$ is trivial\footnote{I thank M.Borovoi
for pointing out the need to assume that $Z(G)$ acts trivially on
all of $F$. If this is not the case, then certain desirable
properties need not hold.},
\item[(ii)] $G=(\img\partial\e)\be\cdot\be Z(G)$, and
\item[(iii)] the map $\partial^{\,\prime}\colon
Z(F)\ra Z(\img\partial\e)$ induced by $\partial$ is surjective.
\end{enumerate}
\end{definition}

Clearly, an abelian crossed module is quasi-abelian.

\begin{example} The crossed module
$(\widetilde{G}\!\ra\be G)$ considered in Example 2.2(iii) is a
quasi-abelian crossed module on $S_{\e\rm{fl}}$ (but perhaps not on
$S_{\e\rm{\acute{e}t}}$ if $\mu$ is non-smooth). Indeed, the induced
action of $Z(G)$ on $\widetilde{G}$ is trivial, $G=G^{\der}Z(G)$ by
\cite{SGA3}, XXII, 6.2.3, and $\partial^{\,\prime}\colon
Z\big(\widetilde{ G}\e\big)=\widetilde{
G}\times_{G^{\lbe\lbe\der}}\lbe\lbe Z( G^{\der})\ra Z( G^{\der})$ is
a surjective map of fppf sheaves.
\end{example}

Definition 3.2 is motivated by

\begin{proposition} Let $( F\!\be\ra\be G)$ be a quasi-abelian
crossed module on $E$. Then  the embedding of crossed modules
\eqref{emb} is a quasi-isomorphism.
\end{proposition}
\begin{proof} Clearly, conditions (i) and (ii) of Definition 3.2
imply that \eqref{emb} is defined.
Further, since $Z(\img\partial\e)=\img\partial\e\cap\e Z( G)$,
there exists a canonical exact sequence
$$
0\ra\cok\partial^{\e\prime}\ra\cok\partial_{Z}\overset{f}\ra
\cok\partial\ra 0,
$$
where the map $f$ is induced by \eqref{emb}. Now, since
$\cok\partial^{\e\prime}=0$ by condition (iii) of Definition 3.2,
$f$ is an isomorphism. Since $\krn\partial_{Z}=\krn\partial$ as
noted above, the proof is complete.
\end{proof}

\begin{remark} Conditions (i) and (ii) of Definition 3.2 imply that
$(Z(F)\!\be\ra\be Z(G))$ coincides with the {\it center} of $(
F\!\be\ra\be G)$, as defined in \cite{GL}, p.171. Thus, the
following is a seemingly more general version of Definition 3.2: a
crossed module is quasi-abelian if it is quasi-isomorphic to its
center.
\end{remark}

\begin{corollary} Let $(F\!\be\ra\be G)$ be a quasi-abelian
crossed module on $E$. Then the maps \eqref{bij} are bijections.\qed
\end{corollary}

When $(F\!\be\ra\be G)$ is quasi-abelian, there exists a rather
useful variant of \eqref{long}. Namely, the short exact sequence of
complexes
$$
0\ra (Z(F)\overset{\partial^{\e\prime}}\twoheadrightarrow
Z(\img\partial\e))\ra(Z(F)\overset{\partial_{Z}}\longrightarrow
Z(G))\ra(0\ra\cok\partial\e)\ra 0
$$
induces an exact sequence of abelian groups
\begin{equation}\label{kamb}
\dots\ra H^{\e i-1}(\cok\partial\e)\ra H^{\le
i+1}(\krn\partial\e)\ra H^{\e i}_{\rm{ab}}(F\!\ra\be
G)\overset{t^{(i)}_{\rm{ab}}}\longrightarrow H^{\e
i}(\cok\partial\e)\ra\dots.
\end{equation}
The map $t^{(i)}_{\rm{ab}}$ was first considered in \cite{Bor2},
\S6.1, when $i=2$. We will write $c^{(i)}\colon H^{\le i}(Z(G))\ra
H^{\le i}(\cok\partial\e)$ for the canonical map induced by the
projection $Z(G)\ra\cok\partial_{Z}=\cok\partial$. Then $c^{(i)}$
factors as
\begin{equation}\label{kamb2}
H^{\le i}(Z(G))\overset{j^{(i)}}\longrightarrow H^{\e
i}_{\rm{ab}}(F\!\ra\be G)\overset{t^{(i)}_{\rm{ab}}}\longrightarrow
H^{\e i}(\cok\partial\e),
\end{equation}
where $j^{(i)}$ is the map appearing in \eqref{long}.

Next, if $(F\!\overset{\partial}\ra\be G)$ is a quasi-abelian
crossed module on $E$, there exists an exact commutative diagram
$$
\xymatrix{ 1\ar[r] & \krn\partial^{\,\prime}\ar[r]\ar@{=}[d] &
Z(F)\ar[r]^{\partial^{\,\prime}}\ar@{^{(}->}[d] & Z(\img\partial\e)\ar[r]
\ar@{^{(}->}[d] & 1\\
1\ar[r] & \krn\partial\ar[r] & F\ar[r]^{\partial} &
\img\partial\ar[r] & 1.}
$$
Consequently, $\partial$ induces an isomorphism
$$
F/Z(F)=\text{Inn}(F)\simeq\img\partial/Z(\img\partial\e)=
\img\partial/\e\img\partial\e\cap\e
Z(G).
$$
On the other hand, by Definition 3.2(ii),
$$
\img\partial/\e\img\partial\e\cap\e Z(G)\simeq\img\partial\e
Z(G)/Z(G)=G/Z(G)=\text{Inn}(G).
$$
We conclude that $\partial$ induces an isomorphism
$\overline{\partial}\e\colon
\text{Inn}(F)\overset{\sim}\ra\text{Inn}(G)$. Now, by \cite{Gi},
Proposition IV.4.2.12(iii), p.285, the exact commutative diagram
$$
\xymatrix{ 1\ar[r] & Z(F)\ar[r]\ar[d]^(.45){\partial_{Z}} &
F\ar[r]\ar[d]^(.45){\partial} & \text{Inn}(F)\ar[r]
\ar[d]_{\simeq}^(.45){\overline{\partial}}& 1\\
1\ar[r] & Z(G)\ar[r] & G\ar[r]& \text{Inn}(G)\ar[r] & 1}
$$
induces an exact commutative diagram
\begin{equation}\label{big}
\xymatrix{ H^{\le 1}({\rm{Inn}}(F))\ar[r]^{d_{\lbe
F}}\ar[d]_{\simeq}^(.45){\overline{\partial}^{\e(1)}} & H^{\e
2}(Z(F))\ar[d]^(.45){\partial_{Z}^{(2)}}\\
H^{\le 1}({\rm{Inn}}(G))\ar[r]^{d_{G}} & H^{\e
2}(Z(G)),}
\end{equation}
where $d_{F}$ and $d_{G}$ are the second coboundary maps of
[op.cit.], IV.4.2.2, p.280, and the left-hand vertical map is a
bijection by [op.cit.], IV.3.1.6.2, p.250. By [op.cit.], Proposition
IV.5.2.8, p.300, the map $d_{\le G}$ (respectively, $d_{\le F}$) may
be described as follows. Let $p\in H^{\le 1}({\rm{Inn}}(G))$, choose
an ${\rm{Inn}}(G)$-torsor $P$ representing $p$ and let
$(H,u\colon\text{lien}(H)\overset{\sim} \ra\text{lien}(G))$ be the
representative of $\text{lien}(G)$ defined by $P$. Thus $H$ is the
twist of $G$ by $P$, where ${\rm{Inn}}(G)$ acts on $G$ in the
natural way, and $u$ is the isomorphism defined in \cite{Gi}, proof
of Corollary IV.1.1.7.3, p.188, lines 6-8. Then $d_{\le G}(p)$ is
the class of the $Z(G)$-gerbe $\text{BITORS}\e(H,G)(u)$ of
$H$-$G$-bitorsors $Q$ on $E$ such that the isomorphism $\pi(Q)$ of
[op.cit.], Corollary IV.5.2.6, p.298, equals $u^{-1}$.

Now, by [op.cit.], Proposition IV.4.2.8(i), p.283\footnote{When
applying this proposition recall that $H^{2}(Z(G))^{\e\prime}=\{0\}$
by [op.cit.], IV.3.3.2.2, p.257.} (see also [op.cit.], Remark
IV.4.2.10, p.284\footnote{Note that the exact sequence appearing in
[loc.cit.] contains an unfortunate misprint: the ``map" $a^{(2)}$
appearing there is only a relation (as in [op.cit.], Definition
IV.3.1.4, p.248), even when $A$ is central in $B$.}), and
Proposition IV.3.2.6, p.255, there exist exact sequences of pointed
sets
\begin{equation}\label{bn0}
H^{\le 1}(G)\overset{b^{\e(1)}_{\lbe G}}\longrightarrow H^{\le
1}({\rm{Inn}}(G)) \overset{d_{G}}\longrightarrow H^{\e 2}(Z(G))
\end{equation}
and
\begin{equation}\label{bn}
H^{\le 1}(G)\overset{b^{\e(1)}_{\lbe G}}\longrightarrow H^{\le
1}({\rm{Inn}}(G)) \overset{n_{G}}\longrightarrow H^{\e
2}(G)^{\e\prime}\ra 1,
\end{equation}
where $b_{G}\colon G\ra {\rm{Inn}}(G)$ is the canonical map and the
map $n_{G}$ is defined as follows: if $p$ and $(H,u)$ are as above,
then $n_{G}(p)\in H^{\e 2}(G)^{\e\prime}$ is the class of the gerbe
$\text{TORS}\e(H)$ of $H$-torsors on $E$, which is a $G$-gerbe via
$u$.

Next, condition (ii) of Definition 3.2 implies, in fact, that
${\rm{Cent}}_{ G}(\img\partial\e)=Z(G)$. We conclude that the
morphism of liens $\text{lien}(\partial\e)\colon
\text{lien}(F)\ra\text{lien}(G)$ satisfies the hypotheses of
\cite{Gi}, Proposition IV.3.1.5, p.249. Consequently, $\partial$
induces a map $\partial^{\e(2)}\colon H^{2}(F)\ra H^{2}(G)$ which
maps $H^{2}(F)^{\e\prime}$ into $H^{2}(G)^{\e\prime}$. On the other
hand, by [op.cit.], Theorem IV.3.3.3(i), p.257, there exists a
simply-transitive action of $H^{\e 2}(Z(G))$ on $H^{\e 2}(G)$ given
by the map
\begin{equation}\label{act}
H^{\e 2}(Z(G))\times H^{\e 2}(G)\ra H^{\e 2}(G), \quad (x,r)
\mapsto x\cdot r,
\end{equation}
where $x\cdot r$ is the class of the contracted product
$X\overset{Z}\wedge R$, where $Z=Z(G)$ and $X$ and $R$ are
representatives of $x$ and $r$, respectively. By
[op.cit.], Corollary IV.3.3.4(ii), p.258, the map $\partial^{(2)}$
is compatible with $\partial_{Z}^{(2)}$ and the actions \eqref{act},
i.e, the following diagram commutes
\begin{equation}\label{comp}
\xymatrix{ H^{\e 2}(Z(F))\times H^{\e 2}(F)
\ar[d]^{\big(\partial_{Z}^{(2)}\be,\,\partial^{(2)}\big)}
\ar[r]& H^{\e
2}(F)\ar[d]^{\partial^{(2)}}\\
H^{\e 2}(Z(G))\times H^{\e 2}(G)
\ar[r]& H^{\e 2}(G).}
\end{equation}

\begin{proposition} Let $d_{\le G}$ and $n_{\le G}$ be the maps appearing
in \eqref{big} and \eqref{bn}, respectively. Then, for every $p\in
H^{\le 1}({\rm{Inn}}(G))$,
$$
n_{G}(p)=d_{\le G}(p)\cdot\varepsilon_{\lbe G}.
$$
\end{proposition}
\begin{proof} As noted above, $n_{G}(p)$ and $\varepsilon_{\lbe G}$
are represented by $\text{T}H:=\text{TORS}\e(H)$ and
$\text{T}G:=\text{TORS}\e(G)$, respectively. On the other hand,
$d_{\le G}(p)$ is represented by the $Z(G)$-gerbe
$\text{BITORS}(H,G)(u)$. Now, by \cite{Gi}, Theorem IV.3.3.3(ii),
p.257, the unique element $x\in H^{2}(Z(G))$ such that $n_{\le
G}(p)=x\cdot\varepsilon_{\lbe G}$ is represented by the $Z(G)$-gerbe
$\text{HOM}_{G}(\text{T}G,\text{T}H)$. Thus, it suffices to check
that there exists a $Z(G)$-equivalence of $Z(G)$-gerbes
$\text{HOM}_{G}(\text{T}G,\text{T}H)\simeq\text{BITORS}(H,G)(u)$.
This follows from \cite{Gi}, Proposition IV.5.2.5(iii), p.297.
\end{proof}

The proposition has the following corollary, previously noted in
\cite{DD}, p.584.

\begin{corollary} The image of $\e d_{G}\colon H^{1}({\rm{Inn}}(G)) \ra H^{2}(Z(G))$
is the set of all elements $x\in H^{2}(Z(G))$ such that
$x\cdot\varepsilon_{\lbe G}\in H^{2}(G)$ is neutral. In particular,
$d_{G}$ is surjective if, and only if, every class of $H^{2}(G)$ is
neutral.
\end{corollary}
\begin{proof} This follows from the proposition, the existence of
\eqref{act} and the exactness of \eqref{bn}.
\end{proof}

Now, for $i=0,1$, we define the $i$-th {\it abelianization map}
\begin{equation}\label{abmaps}
{\rm{ab}}^{i}\colon H^{\le i}(G)\ra H^{\e i}_{\rm{ab}\,}
(F\!\ra\be G)
\end{equation}
as the composite
$$
H^{\le i}(G)\overset{\psi_{i}}\longrightarrow H^{\e i}(F\!\ra\be G)
\overset{\varphi_{i}^{-1}}\longrightarrow  H^{\e i}_{\rm{ab}\,}
(F\!\ra\be G),
$$
where $\psi_{i}$ is the map of Proposition 2.4 and $\varphi_{i}$ is
the bijection \eqref{bij}. Further, let $\delta_{\le 0}\colon H^{\e
0}_{\rm{ab}\,}(F\!\ra\be G)\ra H^{1}(F)$ be the composite
$$
H^{\e 0}_{\rm{ab}\,}(F\!\ra\be G)\overset{\varphi_{\le
0}}\longrightarrow H^{\e 0}(F\!\ra\be G)\overset{\delta_{\le
0}^{\e\prime}}\longrightarrow H^{1}(F),
$$
where $\delta_{\le 0}^{\e\prime}$ is the map described in Remark
2.5(a). Thus $\delta_{\le 0}$ maps the class of a
$(Z(F),Z(G))$-torsor $(Q,t)$ to the class of the $F$-torsor
$Q\be\wedge^{Z}\be F$, where $Z=Z(F)$.

\begin{remarks} \indent
\begin{enumerate}
\item[(a)] If $(F\!\overset{\partial}\ra\be G)$ is a quasi-abelian crossed
module on $\,E$ such that $\partial$ is {\it surjective}, then
$H^{\e 1}_{{\rm{ab}}}(F\!\ra\be G)$ can be identified with $H^{\e
2}(\krn\partial\e)$ (see Example 2.3(ii)). Under this
identification, ${\rm{ab}}^{1}$ corresponds to the coboundary map
$H^{\le 1}(G)\ra H^{\e 2}(\krn\partial\e)$ induced by the central
extension $1\ra\krn\partial\ra F\ra G\ra 1$.
\item[(b)] The right action of $H^{\e 0}(F\!\ra\be G)$ on
$H^{1}(F)$ described in Remark 2.5(c) induces, via $\varphi_{\le
0}$, a right action of $H^{\e 0}_{\rm{ab}\,}(F\!\ra\be G)$ on
$H^{1}(F)$ which can be described as follows. If $p\in H^{\le 1}(F)$
is represented by an $F$-torsor $P$ and $c\in H^{\e
0}_{\rm{ab}\,}(F\!\ra\be G)$ is represented by a
$(Z(F),Z(G))$-torsor $(Q,t)$, then $p\cdot c\in H^{\le 1}(F)$ is
represented by the $F$-torsor $P\be\wedge^{\be F}\be(Q\wedge^{\be
Z}\be F)\simeq P\wedge^{\be Z}\be Q$, where $Z=Z(F)$ (for the
isomorphism, see \cite{Gi}, III,  1.3.1.3, p.115, 1.3.5, p.116 and
2.4.5, p.149). As in Remark 2.5(c), the above action is compatible
with inverse images and with the map $\delta_{\le 0}$. In
particular, the given action is transitive if, and only if, the map
$\delta_{\le 0}$ is surjective.
\end{enumerate}
\end{remarks}

The following statement is immediate from Proposition 2.4 and the
definitions of ${\rm{ab}}^{i}$ and $\delta_{\le 0}$.

\begin{proposition} There exists an exact sequence of pointed sets
$$
\begin{array}{rcl}
1&\ra& H^{\e -1}(F\!\ra\be G)\ra H^{\le 0}(F)\ra H^{\le
0}(G)\overset{{\rm{ab}}^{0}}\longrightarrow H^{\e
0}_{{\rm{ab}}\,}( F\!\ra\be G)\overset{\delta_{\le 0}}\longrightarrow H^{\le 1}(F)\\
&\overset{\partial^{(1)}}\longrightarrow & H^{\le 1}(G)
\overset{{\rm{ab}}^{1}}\longrightarrow H^{\e 1}_{{\rm{ab}}\,}
(F\!\ra\be G).\qed
\end{array}
$$
\end{proposition}

\smallskip

\smallskip

We now discuss twisting. Let $P$ be a (right) $G$-torsor. For any
$G$-object $X$ of $E$, let $\!\phantom{.}^{P}\be\lbe X$ be the twist
of $X$ by $P$ \cite{Gi}, Proposition III.2.3.7, p.146. Now let
$\!\!\phantom{.}^{P}\be\lbe\partial\colon \!\phantom{.}^{P}\be\lbe
F\ra \!\phantom{.}^{P}\be\lbe G$ be the twist of $\partial$ by $P$
\cite{Gi}, III.2.3.3.1, p.142. Further, let $\theta_{\lbe P}\colon
H^{\le 1}(G\e)\overset{\sim}\ra H^{\le
1}\big(\!\phantom{.}^{P}\be\lbe G\e\big)$ be the bijection defined
in [op.cit.], Remark III.2.6.3, p.154. If $P^{\e\rm{o}}\!$ denotes
the $G$-torsor opposite to $P$ [op.cit.], III.1.5.5.2, p.122, and
$Q$ represents the class $q\in H^{\le 1}(G\e)$, then $\theta_{\lbe
P}(q)$ is represented by the $\!\phantom{.}^{P}\be\lbe G$-torsor
$Q\lbe\wedge^{\be G}\be P^{\e\rm{o}}$ [op.cit.], Proposition
III.2.6.1(i), p.153. Let $\!\phantom{.}^{P}\be\lbe{\rm{ab}}^{1}$
denote the composite
$$
H^{\le 1}\big(\!\phantom{.}^{P}\be\lbe
G\e\big)\overset{\!\!\!\phantom{.}^{P}\be\lbe
\psi_{_{1}}}\longrightarrow H^{\e 1}\big(\!\phantom{.}^{P}\be\lbe
F\!\ra\!\!\phantom{.}^{P}\be\lbe
G\e\big)\overset{\!\phantom{.}^{P}\be\lbe\varphi_{1}^{-1}}\longrightarrow
H^{\e 1}_{\rm{ab}} \big(\!\phantom{.}^{P}\be\lbe
F\!\ra\be\phantom{.}^{P}\be\lbe G\e\big)=H^{\e 1}_{\rm{ab}\,}
(F\!\ra\be G).
$$

\begin{proposition} Let $P$ be a $G$-torsor and let $p\in H^{\le 1}(G\e)$ be its
class. Then there exists a commutative diagram
$$
\xymatrix{H^{\le
1}(G\e)\ar[r]^(.43){{\rm{ab}}^{1}}\ar[d]^{\theta_{\lbe P}}& H^{\e
1}_{{\rm{ab}}}(F\!\ra\be G)\ar[d]^{r_{\!\lbe_{P}}}\\
H^{\le 1}\big(\be\!\phantom{.}^{P}\be\lbe
G\e\big)\ar[r]^(.43){\phantom{.}^{P}\be\lbe{\rm{ab}}^{1}}& H^{\e
1}_{{\rm{ab}}}(F\!\ra\be G),}
$$
where $r_{\!\be_{P}}$ is given by
$r_{\!\be_{P}}(x)=x-{\rm{ab}}^{1}(p)$ for $x\in H^{\e
1}_{{\rm{ab}}}(F\!\ra\be G)$.
\end{proposition}
\begin{proof} (C.Demarche) Let $\mathcal C$ be the gr-stack associated to
$(F\!\ra\be G)$, so that $H^{\le 1}(F\!\ra\be G\e)\simeq
H^{1}(\mathcal C)$, and let $\mathcal P$ be a $\mathcal C$-torsor
representing the class in $H^{1}(\mathcal C)$ which corresponds to
$\psi_{1}(p)\in H^{\le 1}(F\!\ra\be G\e)$. Similarly, let $\mathcal
C^{\rm{ab}}$ be the gr-stack associated to $(Z(F)\!\ra\be Z(G))$, so
that $H^{\le 1}(Z(F)\!\ra\be Z(G)\e)\simeq H^{\le 1}(\mathcal
C^{\rm{ab}})$, and let $\mathcal P^{\rm{ab}}$ be a $\mathcal
C^{\rm{ab}}$-torsor representing the class in $H^{\le 1}(\mathcal
C^{\rm{ab}})$ corresponding to
${\rm{ab}}^{1}(p)=\varphi_{1}^{-1}(\psi_{1}(p))\in H^{\le
1}(Z(F)\!\ra\be Z(G)\e)$. Consider the diagram
$$
\xymatrix{H^{\le 1}(G\e)\ar[d]^{\psi_{1}}\ar[r]^(.45){\theta_{\lbe
P}}& H^{\le 1}\big(\!\phantom{.}^{P}\be\lbe
G\,\big)\ar[d]^{\phantom{.}^{P}\be\lbe
\psi_{1}}\\
H^{\le 1}(F\!\ra\be G\e)=H^{\le 1}(\mathcal
C)\ar[d]^{\varphi_{1}^{-1}}\ar[r]^(.45){\theta_{\lbe \mathcal P}}&
H^{\e 1}\big(\!\phantom{.}^{P}\be\lbe
F\!\ra\!\!\phantom{.}^{P}\be\lbe G\e\big)=H^{\le
1}\big(\!\phantom{.}^{\mathcal P}\be\mathcal C\e\big)\ar[d]^{\phantom{.}^{P}
\be\lbe\varphi_{1}^{-1}}\\
H^{\le 1}(Z(F)\!\ra\be Z(G)\e)=H^{\le 1}(\mathcal
C^{\rm{ab}})\ar[d]^{=}\ar[r]^(.45){\theta_{\lbe \mathcal
P^{\rm{ab}}}}& H^{\le 1}\big(Z\big(\!\phantom{.}^{P}\be
F\big)\!\ra\be Z\big(\!\phantom{.}^{P}\be G\big)\e\big)=H^{\le
1}\big(\!\phantom{.}^{\mathcal P^{\rm{ab}}}\be\mathcal C^{\rm{ab}}\big)\ar[d]^{=}\\
H^{\le 1}(Z(F)\!\ra\be Z(G)\e)\ar[r]^(.44){r_{\!_{P}}}&H^{\le
1}(Z(F)\!\ra\be Z(G)\e).}
$$
The left- and right-hand vertical compositions equal ${\rm{ab}}^{1}$
and $\!\phantom{.}^{P}\be\lbe{\rm{ab}}^{1}$, respectively. Now, by
functoriality of twisting \cite{Gi}, III.2.6.3.2, p.155, the above
diagram commutes except perhaps for the bottom square. But by
\cite{Gi}, Remark III.2.6.3, p.154 (which extends easily to
commutative gr-stacks), the bottom square is commutative as well.
This completes the proof.
\end{proof}

By \cite{Gi}, Remark III.3.4.4(2), p.166, there exists an action of
$H^{\e 1}(Z(G))$ on $H^{1}(G)$ given by the map
\begin{equation}\label{act2}
H^{1}(Z(G))\times H^{1}(G)\ra H^{1}(G), \quad (\e p\e,q) \mapsto
p\cdot q,
\end{equation}
where $p\cdot q$ is the class of $P\wedge^{Z} Q$, where $Z=Z(G)$ and
$P$ and $Q$ are representatives of $p$ and $q$, respectively. Now
recall from \eqref{long} the map $j^{\le(1)}\colon H^{1}(Z(G))\ra
H^{\e 1}_{{\rm{ab}}} (F\!\ra\be G)$.

\begin{corollary} For any $p\in H^{1}(Z(G))$ and $q\in H^{1}(G)$,
$$
{\rm{ab}}^{1}\lbe(\e p\cdot q)=j^{\le(1)}(p)+{\rm{ab}}^{1}(q).
$$
\end{corollary}
\begin{proof} (After C.Demarche) Let $Q$ be a representative of $q$.
By the proposition
$$
\phantom{.}^{Q}\lbe{\rm{ab}}^{1}(\theta_{Q}(p\cdot
q))=r_{\!\be_{Q}}({\rm{ab}}^{1}(p\cdot q))={\rm{ab}}^{1}(p\cdot
q)-{\rm{ab}}^{1}(q).
$$
On the other hand, by \cite{Gi}, Proposition III.3.4.5(ii), p.167,
$\theta_{Q}(p\cdot q)=\big(\lbe\!\phantom{.}^{Q}i\big)^{(1)}(p)$,
where $\lbe\!\phantom{.}^{Q}i\colon Z(G)\ra \lbe\!\phantom{.}^{Q}G$
is the $Q$-twist of the canonical embedding $i\colon Z(G)\ra G$. The
corollary now follows from the commutativity of the diagram
$$
\xymatrix{H^{\le
1}(Z(G))\ar[r]^(.41){j^{\le(1)}}\ar[d]^{\big(\lbe\!\phantom{.}^{Q}i\big)^{(1)}}&
H^{\e
1}(Z(F)\!\ra\be Z(G))\ar[d]^{\lbe\!\phantom{.}^{Q}\lbe\varphi_{1}}\\
H^{\le 1}(\lbe\!\phantom{.}^{Q}
G)\ar[r]^(.43){\lbe\!\phantom{.}^{Q}\lbe\psi_{1}}& H^{\e
1}(\lbe\!\phantom{.}^{Q}\lbe F\!\ra\be \lbe\!\phantom{.}^{Q}\lbe
G).}
$$
\end{proof}

The following result is similar to \cite{Gi}, Lemma III.3.3.4,
p.163.

\begin{lemma} Let $P$ be an $F$-torsor and let
$Q=P\wedge^{\be F}G$. Then there exists an exact commutative diagram
\begin{equation}\label{bigdiag}
\xymatrix{H^{\le 0}(G)\ar[r]^(.4){\rm{ab}^{0}} &H^{\e
0}_{{\rm{ab}}\,}(F\!\ra\be G)\ar[r]^(.6){\delta_{0}}& H^{\le
1}(F)\ar[r]^(.5){\partial^{(1)}}\ar[d]^{\theta_{\lbe P}}& H^{\le
1}(G)\ar[r]^(.45){\rm{ab}^{1}}\ar[d]^{\theta_{\lbe Q}}&H^{\e
1}_{{\rm{ab}}\,}(F\!\ra\be G)\ar@{=}[d]\\
H^{\le 0}\big(\be\!\phantom{.}^{Q}\lbe\lbe
G\big)\ar[r]^(.43){\!\phantom{.}^{Q}\be\lbe\rm{ab}^{0}} & H^{\e
0}_{{\rm{ab}}\,}(F\!\ra\be
G)\ar[r]^(.55){\!\phantom{.}^{P}\be\lbe\delta_{0}}& H^{\le
1}\big(\be\!\phantom{.}^{P}\be\lbe
F\big)\ar[r]^{\!\phantom{.}^{P}\be\lbe\partial^{(1)}} & H^{\le
1}(\!\phantom{.}^{Q}\lbe
G)\ar[r]^(.43){\!\phantom{.}^{Q}\lbe\rm{ab}^{1}}&H^{\e
1}_{{\rm{ab}}\,}(F\!\ra\be G).}
\end{equation}
Furthermore, if $p\in H^{1}(F)$ is the class of $P$ and $c$ is any
class in $H^{\e 0}_{{\rm{ab}}\,}(F\!\ra\be G)$, then
$$
\theta_{\lbe\lbe P}(\e p\cdot
c)=\!\phantom{.}^{P}\be\be\delta_{0}(c).
$$
\end{lemma}
\begin{proof} The commutativity of the left-hand square is
\cite{Gi}, III.2.6.3.2, p.155, and that of right-hand square is a
particular case of Proposition 3.11. Now, if $c\in H^{\e
0}_{{\rm{ab}}\,}(F\!\ra\be G)$ is represented by a
$(Z(F),Z(G))$-torsor $(Q,t)$ then, by Remark 3.9(b),
$\theta_{\lbe\lbe P}(\e p\cdot c)$ is represented by the
$\!\phantom{.}^{P}\be\lbe F$-torsor $\big(Q\wedge^{Z}\be
P\big)\wedge^{F}P^{0}\simeq Q\wedge^{Z}\!\phantom{.}^{P}\be\lbe F$,
where $Z=Z(F)$. See \cite{Gi}, Corollary III.1.6.5(1), p.125. This
completes the proof.
\end{proof}

\begin{proposition} Let $(F\!\overset{\partial}\ra\be G)$ be a quasi-abelian
crossed module on $E$.
\begin{enumerate}
\item[(a)] Let $p\in H^{1}(F)$, let $P$ be an $F$-torsor representing $p$ and let
$Q=P\wedge^{\be F}G$. Then the stabilizer of $p$ in $H^{\e
0}_{{\rm{ab}}\,}(F\!\ra\be G)$ is the image of
$\!\phantom{.}^{Q}\lbe{\rm{ab}}^{0}\colon H^{\le
0}\big(\be\!\phantom{.}^{Q}\lbe G\e\big)\ra H^{\e
0}_{{\rm{ab}}}(F\!\ra\be G).$
\item[(b)] The map $\partial^{\e(1)}$ induces an injection
$$
H^{\le 1}(F)/H^{\e 0}_{{\rm{ab}}\,}(F\!\ra\be G)\ra H^{1}(G)
$$
whose image is the kernel of $\e{\rm{ab}}^{1}$.
\end{enumerate}
\end{proposition}
\begin{proof} If $c\in H^{\e 0}_{{\rm{ab}}\,}(F\!\ra\be G)$
satisfies $p\cdot c=p$ then, by the lemma,
$\!\phantom{.}^{P}\be\be\delta_{0}(c)=\theta_{\lbe\lbe P}(\e p\cdot
c)=\theta_{\lbe\lbe P}(p)$ is the unit class of $H^{\le
1}\big(\be\!\phantom{.}^{P}\be\lbe F\e\big)$ \cite{Gi}, Remark
III.2.6.3, p.154, whence $c$ is in the image of
$\phantom{.}^{Q}\lbe{\rm{ab}}^{0}$ by the exactness of the bottom
row of diagram \eqref{bigdiag}. Assertion (a) follows. That the
image of the map in (b) is the kernel of $\e{\rm{ab}}^{1}$ is
immediate from the exactness of the top row of \eqref{bigdiag}. To
prove its injectivity, let $p,p^{\e\prime}\in H^{1}(F)$ be such that
$\partial^{\e(1)}(p)=\partial^{\e(1)}(p^{\e\prime})$. Then the
commutativity of \eqref{bigdiag} and the exactness of its bottom row
show that $\theta_{\lbe
P}(p^{\e\prime})=\phantom{.}^{P}\be\be\delta_{0}(c)=\theta_{\lbe\lbe
P}(\e p\cdot c)$ for some $c\in H^{\e 0}_{{\rm{ab}}\,}(F\!\ra\be
G)$. Thus $p^{\e\prime}=p\cdot c$, which completes the proof.
\end{proof}

\begin{corollary} Let $P$ be a $G$-torsor and let $p\in H^{1}(G)$ be
its class. Then the map $\theta_{\lbe P}^{-1}\circ
\!\phantom{.}^{P}\be\partial^{\e(1)}\colon
H^{1}\big(\be\!\phantom{.}^{P}\be\lbe F\e\big)\ra H^{1}(G)$ induces
a bijection
$$
H^{1}\big(\be\!\phantom{.}^{P}\! F\e\big)/H^{\e
0}_{{\rm{ab}}\,}(F\!\ra\be G)\overset{\!\sim}\ra\{\, p^{\e\prime}\in
H^{1}(G)\colon {\rm{ab}}^{1}(p^{\e\prime})={\rm{ab}}^{1}(p)\}.
$$
\end{corollary}
\begin{proof} Injectivity follows from the injectivity of  $\theta_{\lbe P}^{-1}$ and
part (b) of the proposition applied to the crossed module
$\big(\be\!\phantom{.}^{P}\lbe\be
F\!\overset{\!\!\phantom{.}^{P}\be\lbe\partial}\lra\!\be
\!\phantom{.}^{P}\lbe\be G\big)$. To prove surjectivity, let
$p^{\e\prime}\in H^{1}(G)$ be such that
${\rm{ab}}^{1}(p^{\e\prime})={\rm{ab}}^{1}(p)$ and let
$P^{\e\prime}$ and $P$ be representatives of $p^{\e\prime}$ and $p$,
respectively. Then, by Proposition 3.11,
$$
\!\phantom{.}^{P}\lbe\!{\rm{ab}}^{1}(\theta_{\lbe
P}(P^{\e\prime}))=r_{_{\!\!
P}}({\rm{ab}}^{1}(p^{\e\prime}))={\rm{ab}}^{1}(p^{\e\prime})-{\rm{ab}}^{1}(p)=0.
$$
Thus, by Proposition 3.10 applied to
$\big(\be\!\phantom{.}^{P}\lbe\be
F\!\overset{\!\!\phantom{.}^{P}\be\lbe\partial}\lra\!\be
\!\phantom{.}^{P}\lbe\be G\big)$, $\theta_{\lbe P}(P^{\e\prime})$ is
in the image of $\!\phantom{.}^{P}\be\lbe\partial^{\e(1)}$. This
completes the proof.
\end{proof}

Recall from \eqref{big} and \eqref{bn0} the maps
$\overline{\partial}^{\e(1)}$ and $b^{\e(1)}_{\lbe G}$ and let
$\tilde{b}_{\lbe G}^{\e(1)}\colon H^{\le 1}(G)\ra H^{\le
1}({\rm{Inn}}(F))$ be the composition
\begin{equation}\label{btilde}
H^{\le 1}(G)\overset{b_{\lbe G}^{\e(1)}}\longrightarrow H^{\le
1}({\rm{Inn}}(G))\overset{\overline{\partial}^{\e(1)-1}}
\longrightarrow H^{\le 1}({\rm{Inn}}(F)).
\end{equation}
Recall also from \eqref{long} the map $\pi^{\e(1)}\colon H^{\e
1}_{{\rm{ab}}}(F\!\ra\be G)\ra  H^{\e 2}(Z(F))$.

\begin{lemma} The diagram
$$
\xymatrix{H^{\le 1}(G)\ar[r]^(.45){{\rm{ab}}^{1}} \ar[d]^(.45)
{\tilde{b}^{\e(1)}_{\lbe G}}& H^{\e 1}_{{\rm{ab}}}(F\!\ra\be G)
\ar[d]^(.45){\pi^{\e(1)}}\\
H^{\le 1}({\rm{Inn}}(F))\ar[r]^{d_{\lbe F}} & H^{\e 2}(Z(F))}
$$
commutes.
\end{lemma}
\begin{proof} (After C.Demarche) Let $p$ be a class in $H^{1}(G)$, let $P$ be a
$G$-torsor representing $p$ and let $Q$ be an $\text{Inn}(F)$-torsor
representating $\tilde{b}^{\e(1)}_{\lbe G}\lbe(p)$. Then
$\big(d_{\lbe F}\circ \tilde{b}^{\e(1)}_{\lbe G}\big)(p)$ is
represented by the $Z(F)$-gerbe $K(Q)$ of liftings of $Q$ to $F$
\cite{Gi}, IV.4.2.2, p.280. Now recall the gr-stack $\mathcal
C^{\rm{ab}}$ associated to $(Z(F)\!\ra\be Z(G))$, so that $H^{\e
1}_{{\rm{ab}}}(F\!\ra\be G)\simeq H^{\le 1}(\mathcal C^{\rm{ab}})$,
and let $\mathcal P^{\rm{ab}}$ be a $\mathcal C^{\rm{ab}}$-torsor
representing the class in $H^{\le 1}(\mathcal C^{\rm{ab}})$
corresponding to ${\rm{ab}}^{1}(p)$. Then
$\pi^{\e(1)}({\rm{ab}}^{1}(p))$ is represented by the $Z(F)$-gerbe
$K(\mathcal P^{\rm{ab}})$ of liftings of $\mathcal P^{\rm{ab}}$ to
$Z(G)$ (see \eqref{zseq}). Thus, by \cite{Gi}, Corollary IV.2.2.7,
p.216, it suffices to define a $Z(F)$-morphism of $Z(F)$-gerbes
$m\colon K(Q)\ra K(\mathcal P^{\rm{ab}})$. Let $R$ be a lift of $Q$
to $F$, let $r\in H^{1}(F)$ be its class and set
$t=\partial^{\e(1)}(r)\in H^{1}(G)$. The commutativity of the
diagram
$$
\xymatrix{H^{\le 1}(F\le)\ar[r]^(.45){\partial^{\e(1)}} \ar[d]^(.45)
{b_{\lbe F}^{\e(1)}}& H^{\e 1}(G\le)
\ar[d]^(.45){b_{\lbe G}^{\e(1)}}\\
H^{\le 1}({\rm{Inn}}(F))\ar[r]^{\overline{\partial}^{\e(1)}} & H^{\e
1}({\rm{Inn}}(G))}
$$
shows that $b^{(1)}_{\lbe
G}\lbe(t)=\overline{\partial}^{\e(1)}\big(b_{F}^{\e(1)}(r)\big)=
\overline{\partial}^{\e(1)}(\tilde{b}^{\e(1)}(p))=b^{\e(1)}_{\lbe
G}\lbe(p)$. Thus, by \cite{Gi}, Proposition III.3.4.5, (iii) and
(iv), p.167, there exists a class $z\in H^{1}(Z(G))$, which is
uniquely determined modulo the image of the coboundary map
$\!\phantom{.}^{P}\! d_{\le G}\colon
H^{0}\big(\be\!\phantom{.}^{P}\text{Inn(G)}\big)\ra H^{1}(Z(G))$,
such that $p=z\cdot t$. Note that, since $\!\phantom{.}^{P}\! d_{\le
G}$ factors as
$$
H^{0}\big(\!\phantom{.}^{P}\lbe\text{Inn(G)}\big)\overset
{\big(\!\!\phantom{.}^{P}\overline{\partial}^{\e(0)}\big)^{-1}}\lra
H^{0}\big(\be\!\phantom{.}^{P}
\lbe\text{Inn(F)}\big)\overset{\!\!\phantom{.}^{P}\! d_{\lbe F}}\lra
H^{1}(Z(F))\overset{\partial_{Z}^{(1)}}\ra H^{1}(Z(G)),
$$
$z$ is uniquely determined modulo the image of
$\partial_{Z}^{\e(1)}\circ\!\phantom{.}^{P}\! d_{\lbe F}$. Let $X$
be a $Z(G)$-torsor representing $z$. Since
${\rm{ab}}^{1}(t)={\rm{ab}}^{1}(\partial^{\e(1)}(r))=0$, Corollary
3.12 yields $j^{\e(1)}(z)={\rm{ab}}^{1}(z\cdot t)={\rm{ab}}^{1}(p)$.
Thus $X$ is a lift of $\mathcal P^{\rm{ab}}$ to $Z(G)$ and we set
$m(R)=X$. It is not difficult to check that $\text{lien}(m)$ is the
identity of $Z(F)$, which completes the proof.
\end{proof}

We now recall from \eqref{long} the map $\pi^{\e(1)}\colon H^{\le
1}_{\rm{ab}}(F\!\ra\be G)\ra H^{\e 2}(Z(F))$ and define
\begin{equation}\label{delta}
\delta_{1}\colon H^{\e 1}_{\rm{ab}}(F\!\ra\be G)\ra H^{\le 2}(F)
\end{equation}
by the formula $\delta_{1}(y)=\pi^{\e(1)}(y)\cdot\varepsilon_{\lbe
F}$ for every $y\in H^{\e 1}_{\rm{ab}\,}(F\!\ra\be G)$.

\begin{proposition} A class $y\in H^{\e 1}_{\rm{ab}\,}(F\!\ra\be G)$
is in the image of ${\rm{ab}}^{1}$ if, and only if,
$\delta_{1}(y)\in H^{\e 2}(F)^{\e\prime}$.
\end{proposition}

\begin{proof} Let $q\in H^{1}(G)$. By Lemma 3.16 and Proposition 3.7
(applied to $F$),
$$
\begin{array}{rcl}
\delta_{1}({\rm{ab}}^{1}(q))&=&\pi^{\e(1)}({\rm{ab}}^{1}(q)))
\cdot\varepsilon_{\lbe F}=
d_{F}\lbe\big(\e\tilde{b}^{\e(1)}(q)\big)\cdot\varepsilon_{\lbe F}\\
&=&n_{F}\lbe\big(\e\tilde{b}^{\e(1)}(q)\big) \in H^{\e
2}(F)^{\e\prime}.
\end{array}
$$
Conversely, let $y\in H^{\e 1}_{\rm{ab}\,} (F\!\ra\be G)$ be such
that $\delta_{1}(y)=\pi^{\e(1)}(y) \cdot\varepsilon_{\lbe F}\in
H^{\e 2}(F)^{\e\prime}$. Then $\pi^{\e(1)}(y)=d_{F}(x)$ for some
$x\in H^{\le 1}({\rm{Inn}}(F))$ by the exactness of \eqref{bn} and
Proposition 3.7. Now, by the exactness of \eqref{long} and the
commutativity of \eqref{big},
$$
0=\partial_{Z}^{(2)}(\pi^{\e(1)}(y))=
\partial_{Z}^{(2)}(d_{F}(x))=
d_{\lbe G}\big(\e\overline{\partial}^{\e(1)}(x)\big).
$$
We conclude that $\overline{\partial}^{\e(1)}(x)=b^{\e(1)}(z)$ for
some $z\in H^{1}(G)$ by the exactness of \eqref{bn0}. Thus
$x=\tilde{b}^{\e(1)}(z)$ by the definition of $\tilde{b}^{\e(1)}$
\eqref{btilde}, and Lemma 3.16 yields
$$
\pi^{\e(1)}(y)=d_{F}(x)=d_{F}\lbe\big(\e\tilde{b}^{\e(1)}(z)\big)=
\pi^{\e(1)}({\rm{ab}}^{1}(z)).
$$
Now the exactness of \eqref{long} shows that
$y-{\rm{ab}}^{1}(z)=j^{\e(1)}(w)$ for some $w\in H^{1}(Z(G))$,
whence
$$
y=j^{\e(1)}(w)+{\rm{ab}}^{1}(z)={\rm{ab}}^{1}(w\be\cdot\be z),
$$
by Corollary 3.12. The proof is now complete.
\end{proof}

\begin{proposition} The sequence
$$
H^{\e 1}_{{\rm{ab}}}(F\!\ra\be G)\overset{\delta_{1}}\lra H^{\e
2}(F) \overset{\,\partial^{(2)}}\lra H^{\e 2}(G)
$$
is exact.
\end{proposition}
\begin{proof} Let
$y\in H^{\e 1}_{\rm{ab}\,}(F\!\ra\be G)$. By the definition of
$\delta_{1}$, the commutativity of \eqref{comp} and the exactness of
\eqref{long}, we have
$$
\partial^{\e(2)}(\delta_{1}(y))=\partial_{Z}^{\e(2)}(\pi^{\e(1)}(y))\cdot
\varepsilon_{\lbe G}=0\cdot\varepsilon_{\lbe G}=\varepsilon_{\lbe
G}.
$$
Conversely, assume that $s=z\cdot\varepsilon_{\lbe F}\in H^{\e
2}(F)$ (where $z\in H^{\le 2}(Z(F))$) is such that
$\partial^{\e(2)}(s)=
\partial_{Z}^{\e(2)}(z)\cdot\varepsilon_{\lbe G}=\varepsilon_{\lbe G}$.
Then $\partial_{Z}^{\e(2)}(z)=0$, whence $z=\pi^{\e(1)}(y)$ for some
$y\in H^{\e 1}_{\rm{ab}\,}(F\!\ra\be G)$ by the exactness of
\eqref{long}. Thus $s=\pi^{\e(1)}(y)\cdot\varepsilon_{\lbe F}=
\delta_{1}(y)\in\img\,\delta_{1}$.
\end{proof}

\section{The main theorem}

Let $(F\!\be\overset{\partial}\ra\be G)$ be a quasi-abelian crossed
module on $E$. The \textit{second abelianization map} of
$(F\!\be\ra\be G)$ is the map
\begin{equation}\label{ab2}
{\rm{ab}}^{2}\colon H^{\le 2}(G)\ra H^{\e 2}_{\rm{ab}}(F\!\ra\be G)
\end{equation}
defined as follows: if $s\in H^{\le 2}(G)$ and $x$ is the unique
element of $H^{\le 2}(Z(G))$ such that $s=x\cdot\varepsilon_{\lbe
G}$ \eqref{act}, then ${\rm{ab}}^{2}(s)=j^{\e(2)}(x)$, where
$j^{\e(2)}\colon H^{\le 2}(Z(G))\ra H^{\e 2}_{\rm{ab}\,}(F\!\ra\be
G)$ is the map appearing in \eqref{long}.

 We will also need the map
\begin{equation}\label{t}
t:=t_{\rm{ab}}^{(2)}\circ{\rm{ab}}^{2}\colon H^{\le 2}(G)\ra H^{\e
2}(\cok\partial\e),
\end{equation}
where $t_{\rm{ab}}^{(2)}$ is the map appearing in \eqref{kamb}. If
$s$ and $x$ are as above, then
$t(s)=t_{\rm{ab}}^{(2)}(j^{\e(2)}(x))=c^{(2)}(x)$, where
$c^{(2)}\colon H^{\le 2}(Z(G))\ra H^{\e 2}(\cok\partial\e)$ is the
canonical map induced by the projection
$Z(G)\ra\cok\partial_{Z}=\cok\partial$ (see \eqref{kamb2}).

We now extract from \eqref{long} the subsequence
\begin{equation}\label{extr}
H^{\e 2}_{\rm{ab}\,}(F\!\ra\be G)\overset{\pi^{(2)}}\longrightarrow
H^{\e 3}(Z(F))\overset{\partial_{Z}^{(3)}}\longrightarrow H^{\e
3}(Z(G))\ra\dots.
\end{equation}
Then the following holds.

\begin{proposition} The sequence
$$
H^{\e 2}(F)\overset{\,\partial^{(2)}}\longrightarrow H^{\e
2}(G)\overset{{\rm{ab}}^{2}}\longrightarrow H^{\e
2}_{\rm{ab}\,}(F\!\ra\be G)\overset{\pi^{(2)}}\longrightarrow
H^{\e 3}(Z(F))\overset{\partial_{Z}^{(3)}}\longrightarrow H^{\e 3}(Z(G))\ra\dots.
$$
is exact.
\end{proposition}
\begin{proof} Since \eqref{extr} is exact, we need only check the
exactness of the given sequence at $H^{\e 2}(G)$ and $H^{\e
2}_{\rm{ab}\,}(F\!\ra\be G)$. Exactness at $H^{\e
2}_{\rm{ab}\,}(F\!\ra\be G)$ is not difficult: we have
$\pi^{(2)}\circ\e {\rm{ab}}^{2}=\pi^{(2)}\be\circ j^{\e(2)}=0$, and
if $y\in H^{\e 2}_{\rm{ab}\,}(F\!\ra\be G)$ is such that
$\pi^{(2)}(y)=0$, then (by the exactness of \eqref{long}) there
exists an element $x\in H^{\e 2}(Z(G))$ such that
$y=j^{\e(2)}(x)={\rm{ab}}^{2}(x\cdot\varepsilon_{\lbe
G})\in\img\,{\rm{ab}}^{2}$. To check exactness at $H^{\e 2}(G)$, let
$t=y\cdot\varepsilon_{\lbe F}\in H^{\e 2}(F)$, where $y\in H^{\le
2}(Z(F))$. Then, by the commutativity of \eqref{comp},
$\partial_{Z}^{\e(2)}(y)$ is the unique element $x$ of $H^{\le
2}(Z(G))$ such that $\partial^{\e(2)}(t)=x\cdot \varepsilon_{\lbe
G}$. Consequently ${\rm{ab}}^{2}(\partial^{\e(2)}(t))=j^{\e(2)}(x)=
j^{\e(2)}(\partial_{Z}^{\e(2)}(y))=0$, since \eqref{long} is exact.
On the other hand, if $s\in H^{\le 2}(G)$ is such that
${\rm{ab}}^{2}(s)=j^{\e(2)}(x)=0$, where $x$ is the unique element
of $H^{\le 2}(Z(G))$ such that $s=x\cdot\varepsilon_{\lbe G}$, then
$x=\partial_{Z}^{\e(2)}(y)$ for some $y\in H^{\le 2}(Z(F))$ by the
exactness of \eqref{long}. Thus, again by the commutativity of
\eqref{comp},
$$
s=\partial_{Z}^{\e(2)}(y)\cdot\varepsilon_{\lbe
G}=\partial^{\e(2)}(y \cdot \varepsilon_{\lbe
F})\in\img\,\partial^{\e(2)}.
$$
This completes the proof.
\end{proof}

Combining the above proposition with Propositions 3.10, 3.17 and
3.18, we obtain the main result of this paper:

\begin{theorem} Let $(F\!\be\overset{\partial}\ra\be G)$ be
a quasi-abelian crossed module on $E$. Then
$$
\begin{array}{rcl}
1&\ra& H^{\e -1}(F\!\ra\be G)\ra H^{\le 0}(F)\ra
H^{\le 0}(G)\overset{{\rm{ab}}^{0}}\longrightarrow H^{\e
0}_{{\rm{ab}}\,}( F\!\ra\be G)\overset{\delta_{0}}\longrightarrow H^{\le 1}(F)\\
&\overset{\,\partial^{\e(1)}}\longrightarrow& H^{\le 1}(G)
\overset{{\rm{ab}}^{1}}\longrightarrow H^{\e 1}_{{\rm{ab}}\,}
(F\!\ra\be G) \overset{\delta_{1}}\longrightarrow H^{\e 2}(F)
\overset{\,\partial^{(2)}} \longrightarrow H^{\e
2}(G)\overset{{\rm{ab}}^{2}}\longrightarrow H^{\e
2}_{\rm{ab}\,}(F\!\ra\be G)\\
&\overset{\pi^{(2)}}
\longrightarrow& H^{\e 3}(Z(F))\overset{\partial_{Z}^{(3)}}
\longrightarrow H^{\e 3}(Z(G))\ra\dots
\end{array}
$$
is an exact sequence of pointed sets at every term except $H^{\e
1}_{{\rm{ab}}\,} (F\!\ra\be G)$, where a class $y\in H^{\e
1}_{{\rm{ab}}\,} (F\!\ra\be G)$ is in the image of ${\rm{ab}}^{1}$
if, and only if, $\delta_{1}(y)\in H^{\e 2}(F)^{\e\prime}$.\qed
\end{theorem}

\begin{remark} The exact sequence of the theorem is compatible with
inverse images, i.e., if $u\colon E^{\e\prime}\ra E$ is a morphism
of sites, then there exists an exact commutative diagram
$$
\xymatrix{\dots\ar[r]^(0.45){\partial^{(1)}} &H^{\le 1}(E,
G\e)\ar[r]^(0.43){\rm{ab}^{1}}\ar[d]& H^{\e 1}_{\rm{ab}}(E,F\ra
G\e)\ar[r]^(0.55){\delta_{1}}\ar[d]& H^{\le
2}(E,F)\ar[r]^(0.58){\partial^{(2)}}\ar[d]&\dots\\
\dots\ar[r]^(0.3){(\partial^{(1)})^{\e\prime}} &H^{\le
1}(E^{\e\prime}, u^{\lbe
*}G\e)\ar[r]^(0.4){(\be\rm{ab}^{1}\be)^{\e\prime}} & H^{\e
1}_{\rm{ab}}(E^{\e\prime},u^{\lbe *}F\ra u^{\lbe
*}G\e)\ar[r]^(0.57){\delta_{1}^{\e\prime}}& H^{\le
2}(E^{\e\prime},u^{\lbe
*}F)\ar[r]^(0.7){(\partial^{(2)})^{\e\prime}}&\dots}
$$
whose rows are the exact sequences of the theorem over $E$ and over
$E^{\e\prime}$ and whose vertical maps are given by \eqref{pbck} and
\cite{Gi}, Definition V.1.5.1, p.316. This follows from the
definitions of the various horizontal maps involved, Remark 3.9(b)
and \cite{Gi}, Proposition V.1.5.2(ii), p.317.
\end{remark}

Regarding the map ${\rm{ab}}^{2}$, the following holds.

\begin{lemma} $H^{\e 2}(G)^{\prime}\subset\krn{\rm{ab}}^{2}$.
\end{lemma}
\begin{proof} By the theorem, we need to check that $H^{\e
2}(G)^{\e\prime}$ is contained in the image of $\partial^{\e(2)}$.
Let $s\in H^{\e 2}(G)^{\e\prime}$. By the exactness of \eqref{bn}
and Proposition 3.7, there exists an element $p\in H^{\le
1}({\rm{Inn}}(G))$ such that $s=d_{\le G}(p)\cdot \varepsilon_{\lbe
G}$. On the other hand, by the surjectivity of
$\overline{\partial}^{\e(1)}$ and the commutativity of diagram
\eqref{big}, there exists an element $q\in H^{\le 1}({\rm{Inn}}(F))$
such that
$$
d_{\le G}(p)=d_{\le G}\big(\,
\overline{\partial}^{\e(1)}\lbe(q)\big)=\partial_{Z}^{\e(2)}
(d_{\le F}(q)).
$$
Thus, by the commutativity of \eqref{comp},
$$
s=d_{\le G}(p)\cdot\varepsilon_{\lbe G}=\partial^{\e(2)} (d_{\le
F}(q)\cdot\varepsilon_{\lbe F})\in\img\,\partial^{\e(2)}.
$$
\end{proof}

Since $\partial^{\e(2)}\colon H^{\e 2}(F)\ra H^{\e 2}(G)$ maps
$H^{\e 2}(F)^{\e\prime}$ into $H^{\e 2}(G)^{\e\prime}$, Theorem 4.2
and the above lemma yield inclusions
$$
\partial^{\e(2)}(H^{\e 2}(F)^{\e\prime})\subset H^{\e 2}(G)^{\e\prime}\subset
\krn{\rm{ab}}^{2}=
\partial^{\e(2)}(H^{\e 2}(F)).
$$
Thus, the following is an immediate consequence of Theorem 4.2.

\begin{corollary} Let $(F\!\ra\be G)$ be a quasi-abelian crossed
module on $E$ such that every class of $H^{\le 2}(F)$ is neutral.
Then
\begin{enumerate}
\item[(i)] the abelianization map ${\rm{ab}}^{1}\colon H^{\le
1}(G)\ra H^{\e 1}_{{\rm{ab}}} (F\!\ra\be G)$ is surjective, and
\item[(ii)] a class $s\in H^{\e 2}(G)$ is neutral if, and only if,
${\rm{ab}}_{\lbe G}^{2}(s)=0$.\qed
\end{enumerate}
\end{corollary}

\section{Applications}

Let $S$ be a scheme and let $S_{\e\rm{fl}}$ be the small fppf site
over $S$. An $S$-group scheme $G$ is called {\it reductive}
(respectively, {\it semisimple}) if it is affine and smooth over $S$
and its geometric fibers are  {\it connected} reductive
(respectively, semisimple) algebraic groups \cite{SGA3}, XIX,
Definition 2.7. The composition $\partial\colon
\widetilde{G}\twoheadrightarrow G^{\der}\hookrightarrow G$ defines a quasi-abelian crossed module
$\big(\widetilde{G}\be\overset{\partial}\ra\be G\e\big)$ on
$S_{\e\rm{fl}}$ (see Example 3.3). We have $\krn\partial=\mu$ and
$\cok\partial=G^{\tor}$ as sheaves on $S_{\e\rm{fl}}$, where $\mu$
is the fundamental group of $G$ and $G^{\tor}=G/G^{\der}$ is the
coradical of $G$ \cite{SGA3}, XXII, 6.2. Set $H^{\e
i}_{\rm{ab}}(S_{\e\rm{fl}},G)=H^{\e
i}_{\rm{ab}}(S_{\e\rm{fl}},\widetilde{G}\ra G)$, which will also be
denoted by $H^{\e i}_{\rm{ab}}(G)$ to simplify some statements. If
$G$ has trivial fundamental group (respectively, if $G$ is
semisimple), then $H^{\e i}_{\rm{ab}}(G)=H^{\e i}(G^{\tor}\e)$
(respectively, $H^{\e i}_{\rm{ab}}(G)=H^{\e i+1}(\mu)$). See
Examples 2.3. If $K$ is a field and $G$ is a (connected) reductive
algebraic group over $K$, $H^{\le 1}(K,G)$ will denote the first
Galois cohomology set of $G$. Note that there exists a canonical
bijection $H^{\le 1}(K,G)\simeq H^{\le 1}(K_{\rm{fl}},G)$
\cite{Mi1}, Remark III.4.8(a), p.123\e\footnote{In fact, the natural
map $H^{\le 1}(S_{\rm{\acute{e}t}},G)\ra H^{\e 1}(S_{\e\rm{fl}},G)$
is bijective for any $S$ by the smoothness of $G$. See [loc.cit.].}.
If $G$ is, in addition, commutative and $i\geq 1$, then $H^{\le
i}(K,G)$ will denote the $i$-th Galois cohomology group of $G$.

\smallskip

By \eqref{-1} and Theorem 4.2, the following holds.

\begin{theorem} Let $G$ be a reductive group scheme over a scheme
$S$. Then there exists a sequence of flat (fppf) cohomology sets
$$
\begin{array}{rcl}
1&\longrightarrow&\mu(S)\ra\widetilde{G}(S)\ra G (S)
\overset{\rm{ab}^{0}}\longrightarrow H^{\e 0}_{\rm{ab}}(G
\e)\overset{\delta_{0}}\lra  H^{\le 1}\big(\widetilde{G}\e\big)
\overset{\,\partial^{\e(1)}}\longrightarrow H^{\le 1}(G\e)\\
&\overset{\rm{ab}^{1}}\longrightarrow& H^{\e 1}_{\lbe\rm{ab}}(G
\e)\overset{\delta_{\le 1}}\lra H^{\le 2}\big(\widetilde{G
}\e\big)\overset{\,\partial^{\e(2)}}\longrightarrow H^{\e 2}(G
\e)\overset{\rm{ab}^{\lbe 2}}\longrightarrow H^{\e 2}_{\rm{ab}}(G
\e)\ra H^{\e 3}\big(Z\big(
\widetilde{G}\e\big)\big)\\
&\longrightarrow& H^{\e 3}(Z(G\e))\ra\dots,
\end{array}
$$
which is an exact sequence of pointed sets at every term except
$H^{\e 1}_{{\rm{ab}}} (G)$, where a class $y\in H^{\e
1}_{{\rm{ab}}}(G)$ is in the image of $\,{\rm{ab}}^{1}$ if, and only
if, $\delta_{1}(y)\in H^{\le 2}\big(\widetilde{G
}\e\big)^{\prime}$.\qed
\end{theorem}

\smallskip

Recall that, if $G$ is a reductive group scheme, $G^{\e\rm{ad}}$ is
the standard notation for $\text{Inn}(G)$.

\begin{definition} A scheme $S$ is called {\it of Douai type}
if, for every semisimple and simply-connected $S$-group scheme $G$,
the coboundary map
$$
d_{\le G}\colon H^{1}(S_{\e\rm{fl}},G^{\e\rm{ad}})\ra H^{\le
2}(S_{\e\rm{fl}},Z(G))
$$
(induced by the central extension $1\ra Z(G)\ra G\ra
G^{\e\rm{ad}}\ra 1$) is surjective. Equivalently (see Corollary
3.8), every class of $H^{\le 2}(S_{\e\rm{fl}},G)$ is neutral.
\end{definition}

If $S=\spec A$ is affine, then we will also say that $A$ is of Douai
type.

\smallskip

\begin{remark} By Remark 3.9(a), the map $d_{G}$ appearing in the above
definition can be identified with the first abelianization map
${\rm{ab}}^{1}$ for the adjoint group $G^{\e\rm{ad}}$.
\end{remark}

\begin{examples} The following are examples of schemes of Douai
type.

\begin{enumerate}

\item[(i)] $S=\spec K$, where $K$ is a complete
discretely-valued field with finite residue field. See \cite{Do},
Chapter VI, Theorems 1.4, p.80, 1.6, p.81, and 2.1, p.87. See also
\cite{Do1}, Theorem 1.1 and Remark on p.322.

\item[(ii)] $S=\spec K$, where $K$ is a global field, i.e.,
$K$ is either a number field or a function field in one variable over a
finite field. See \cite{Do}, Theorem VIII.1.2,
p.108. See also \cite{Do2}.

\item[(iii)] $S$ is a nonempty open subscheme of either the spectrum
of the ring of integers of a number field or a smooth, complete and
irreducible curve over a finite (respectively, algebraically closed)
field. See \cite{Do3}, Theorem 1.1 and Remark (b), p.326, and note
that Lemma 1.1 from [op.cit.] holds over any $S$ as above by, e.g.,
\cite{Mi}, proof of Proposition II.2.1, p.202.

\item[(iv)] $S=\spec K$, where $K$ is a function field in one
variable over a quasi-finite field $k$ of positive characteristic
which is algebraic over its prime subfield $k_{\e 0}$ and satisfies
$[k\e\colon\be k_{\e 0}]=\prod p^{\e n_{p}}$, where $n_{p}<\infty$
for every prime $p$. See \cite{Do5}, proof of Proposition 1.1 and
Remark 1.2.

\item[(v)] $S$ is a smooth, complete and irreducible curve
over a quasi-finite field $k$ with function field $K$, where $k$ and
$K$ are as in (iv). See \cite{Do5}, proof of Lemma 3.1 and Remark
3.2(1) and (3).

\item[(vi)] $S=\spec K$, where $K$ is a field of characteristic
zero and of (Galois) cohomological dimension $\leq 2$ such that, for
central simple algebras over finite extensions of $K$, exponent and
index coincide. Indeed, it is shown in \cite{CTGP}, proof of Theorem
2.1(a), that $d_{\le G}\colon H^{1}(K_{\rm{fl}},G^{\e\rm{ad}})\ra
H^{2}(K_{\rm{fl}},Z(G))$ is surjective for every semisimple and
simply-connected $K$-group $G$. Note that, by \cite{J}, examples of
such fields include the ``fields of types (gl), (ll) and (sl)"
considered in \cite{CTGP} (see below for the definitions of types
(gl) and (ll)). See [op.cit.], Theorems 1.3\e -1.5).

\item[(vii)]  $S$ is a regular and integral two-dimensional scheme
equipped with a proper birational morphism $S\ra\spec A$, where $A$
is an excellent, henselian, two-dimensional local domain with
algebraically closed residue field of characteristic $0$. See
\cite{Do7}.

\item[(viii)] $S$ is a projective, smooth and geometrically
irreducible curve over a $p$-adic field. See the forthcoming paper
\cite{Do8}.

\end{enumerate}

\end{examples}

\begin{theorem} Let $S$ be a scheme of Douai type and let $G$ be a
reductive group scheme over $S$.
\begin{enumerate}
\item[(i)] The group $H^{\e 0}_{\rm{ab}}(S_{\e\rm{fl}},G \e)$ acts on the right
on the set $H^{\le 1}\big(S_{\e\rm{fl}},\Gtil\e\big)$ compatibly
with the map $\delta_{\le 0}\colon H^{\e
0}_{\rm{ab}}(S_{\e\rm{fl}},G \e)\ra H^{\le
1}\big(S_{\e\rm{fl}},\Gtil\e\big)$ and there exists an exact
sequence of pointed sets
$$
1\ra H^{\le 1}\big(S_{\e\rm{fl}},\Gtil\e\big)/H^{\e
0}_{\rm{ab}}(S_{\e\rm{fl}},G \e)\overset{\bar{\partial}^{\e(1)}}\lra
H^{\le 1}(S_{\e\rm{fl}},G\e)\overset{\rm{ab}^{1}}\longrightarrow
H^{\e 1}_{\lbe\rm{ab}}(S_{\e\rm{fl}},G \e)\ra 1,
$$
where the map $\bar{\partial}^{\e(1)}$ (which is induced by
$\partial^{\e(1)}$) is injective.
\item[(ii)] A class $\xi\in H^{\e 2}(S_{\e\rm{fl}},G)$ is neutral
if, and only if, ${\rm{ab}}^{2}_{\lbe G}(\xi)=0$.
\end{enumerate}
\end{theorem}
\begin{proof} The surjectivity of $\rm{ab}^{1}$ and assertion (ii) are
immediate from Corollary 4.5 and Definition 5.2. The action
mentioned in (i) is defined in Remark 3.9(b), and the exactness of
the sequence in (i) follows from Proposition 3.14(b) and the
surjectivity of $\rm{ab}^{1}$.
\end{proof}

\begin{remarks} \indent

\begin{enumerate}

\item[(a)] The exact sequence in part (i) of the theorem is
compatible with inverse images, i.e., if $S^{\e\prime}\ra S$ is a
morphism of schemes of Douai type, then the diagram
$$
\xymatrix{1\ra H^{\e 1}\be\big(S_{\e\rm{fl}},\Gtil\e\big)/H^{\e
0}_{\rm{ab}}(S_{\e\rm{fl}},G\e)\ar[r]\ar[d]& H^{\e
1}\be(S_{\e\rm{fl}},G\e)\ar[r]\ar[d]& H^{\e
1}_{\rm{ab}}(S_{\e\rm{fl}},G\e)\ar[d]\ar[r]&1\\
1\ra H^{\e 1}\be\big(S^{\e\prime}_{\e\rm{fl}},\Gtil\e\big)/H^{\e
0}_{\rm{ab}}(S^{\e\prime}_{\rm{fl}},G\e)\ar[r]& H^{\e
1}\be(S^{\e\prime}_{\e\rm{fl}},G\e)\ar[r]& H^{\e
1}_{\rm{ab}}(S^{\e\prime}_{\rm{fl}},G\e)\ar[r]&1}
$$
commutes. This follows from Remark 4.3, and the fact that the action
of $H^{\e 0}_{\rm{ab}}(S_{\e\rm{fl}},G\e)$ on $H^{\e
1}\be\big(S_{\e\rm{\acute{e}t}},\Gtil\e\big)$ is compatible with
inverse images. See Remark 3.9(b).

\item[(b)] Let $S$ and $G$ be as in the theorem and let $\mathcal C^{\rm{ab}}$
be the gr-stack associated to $\big(Z\big(\Gtil\big)\!\ra\be
Z(G)\big)$, so that there exists a bijection $H^{\e
1}_{{\rm{ab}}}(S_{\e\rm{fl}},G\e)\simeq H^{\le 1}(\mathcal
C^{\rm{ab}})$. For each class $\xi\in H^{\e
1}_{{\rm{ab}}}(S_{\e\rm{fl}},G\e)$, let $\mathcal
P^{\e\rm{ab}}_{\!\!_{\xi}}$ be a $\mathcal C^{\rm{ab}}$-torsor representing
its image in $H^{\le 1}(\mathcal C^{\rm{ab}})$ and let $P_{\!\be_{\xi}}$ be
a lift of $\mathcal
P^{\e\rm{ab}}_{\!\!_{\xi}}$ to $G$. Then part (i) of
the theorem and Corollary 3.15 yield a (non-canonical) bijection
$$
H^{\e 1}(S_{\e\rm{fl}},G\e)\simeq\coprod_{\xi\e\in\e H^{\e
1}_{{\rm{ab}}}\be(S_{\rm{fl}},G\e)}H^{\le
1}\big(S_{\e\rm{fl}},\!\!\phantom{.}^{P_{\!\!_{\xi}}}\lbe \Gtil\e\big)/H^{\e
0}_{\rm{ab}}(S_{\e\rm{fl}},G \e).
$$

\item[(c)] The theorem holds if $S$ is any scheme
and $G$ is a reductive group scheme over $S$ such that every class
of $H^{\le 2}\big(S_{\e\rm{fl}},\widetilde{G}\e\big)$ is neutral.
For example, if $S$ is a ruled surface of the type considered in
\cite{Do4}, Corollary 3.14, p.76\e\footnote{ We do not know if these
schemes are of Douai type.}, and $G$ is {\it split}, then the
theorem holds for $G$ (and part (i) generalizes [op.cit.], Corollary
3.15). However, we have chosen to work with schemes $S$ of Douai
type because we want our statements to apply uniformly to all
reductive group schemes over $S$.

\item[(d)] By Example 5.4(iii), the theorem holds if $S$ is
the spectrum of the ring of integers of a number field. Thus
Corollary 1.2 of the Introduction is contained in part (i) of the
theorem.
\end{enumerate}
\end{remarks}

Let $S$ be a scheme and let $L$ be an $S_{\e\rm{fl}}$-lien which is
locally represented by a reductive group scheme over $S$. Let
${\rm{ab}}_{L}^{2}\colon H^{\e 2}(S_{\e\rm{fl}},L)\ra
H^{2}_{\rm{ab}}(S_{\e\rm{fl}},L)$ be the map defined in \cite{Do6},
p.23. It is not difficult to check that if $L=\text{lien}(G)$ is
represented by a group $G$ of the topos $\widetilde{S}_{\e\rm{fl}}$,
${\rm{ab}}_{L}^{2}$ coincides with the map ${\rm{ab}}_{\lbe G}^{2}$
considered above. The following result generalizes \cite{Bor2},
Theorem 5.5.

\begin{corollary} Let $S$ be a scheme of Douai type and
let $L$ be an $S_{\e\rm{fl}}$-lien which is locally represented by a
reductive $S$-group scheme. Then a class $\xi\in H^{\e
2}(S_{\e\rm{fl}},L)$ is neutral if, and only if,
${\rm{ab}}_{L}^{2}(\xi)=0$.
\end{corollary}
\begin{proof} Assume that $L$ is locally represented by a reductive $S$-group
scheme $G$. By \cite{Do6}, Proposition 1.2, p.22 (see also
\cite{Do1}, Lemma 1.1, and \cite{Do}, V.3.1, p.74), $L$ is, in fact,
globally represented by a quasi-split form $G_{\be L}$ of $G$, i.e.,
$L\simeq\text{lien}(G_{\be L})$. The result now follows by applying
part (ii) of the proposition to $G_{\be L}$.
\end{proof}

\begin{theorem} Let $K$ be a field of Douai type and of Galois cohomological
dimension $\leq 2$. Let $G$ be a (connected) reductive algebraic
group over $K$.
\begin{enumerate}
\item[(i)] If $H^{1}(K,H)$ is trivial for every semisimple and
simply-connected $K$-group $H\e$\footnote{By Serre's conjecture II
(see \cite{GS}), this is expected to follow from the hypothesis
${\rm{cd}}\, K\leq 2$.}, then the first abelianization map
$\,{\rm{ab}}^{1}\colon H^{1}(K,G)\ra H^{\e
1}_{\rm{ab}}(K_{\rm{fl}},G)$ is bijective.

\item[(ii)] There exists an exact sequence of pointed sets
$$
1\ra H^{2}(K_{\rm{fl}},G)^{\prime}\ra H^{2}(K_{\rm{fl}},G)
\overset{\! t}\ra H^{\e
2}(K,G^{\tor}\e) \ra 1,
$$
where $t$ is the map \eqref{t}.
\end{enumerate}
\end{theorem}
\begin{proof} (i) The hypothesis and Corollary 3.15 show that
${\rm{ab}}^{1}$ is injective. Since it is surjective by Theorem
5.5(i), it is in fact bijective.

(ii) Since $K$ has cohomological dimension $\leq 2$ and both $\mu$
and $Z\big(\widetilde{G}\e\big)$ are commutative and finite
$K$-group schemes, $H^{\e i}\lbe\big(K_{\rm{fl}},\mu)=H^{\e
i}\lbe\big(K_{\rm{fl}}, Z\big(\widetilde{G}\e\big)\big)=0$ for every
$i\geq 3$ by \cite{Sh}, Theorem 4, p.593. Thus Theorem 5.1 and
Theorem 5.5(ii) yield an exact sequence
$$
1\ra H^{2}(K_{\rm{fl}},G)^{\prime}\ra
H^{2}(K_{\rm{fl}},G)\overset{\rm{ab}^{2}}\longrightarrow H^{\e
2}_{\rm{ab}}(K_{\rm{fl}},G) \ra 1.
$$
On the other hand, \eqref{kamb} shows that $t_{\rm{ab}}^{(2)}\colon
H^{\e 2}_{\rm{ab}}(K_{\rm{fl}},G)\ra H^{\e 2}(K,G^{\tor})$ is an
isomorphism. Thus, since $t=t_{\rm{ab}}^{(2)}\circ\rm{ab}^{2}$, the
sequence of the statement is indeed exact.
\end{proof}

\begin{remarks}\indent

\begin{enumerate}

\item[(a)] A field which is either a complete and discretely-valued field with finite
residue field or a global field without real primes satisfies the
hypotheses, and therefore the conclusions, of the theorem. See
Examples 5.4(i) and (ii), \cite{PR}, Theorems 6.4 and 6.6, pp.284
and 286, \cite{BT}, Theorem 4.7(ii), and \cite{Ha}, Theorem A,
p.125.

\item[(b)] The conclusion in part (i) of the theorem for
the fields of types (gl), (ll) and (sl) mentioned in Remark 5.4(vi)
was previously established in \cite{BKG}, Theorem 6.7 (provided that
$G$ contains no factors of type $E_{8}$ in the (gl) case).

\item[(c)] Let $K$ be
either a global function field or the completion of such a field at
one of its primes. Let $G$ be a {\it semisimple} algebraic group
over $K$. Then ${\rm{ab}}^{\lbe 1}$ can be identified with the
coboundary map $H^{\le 1}(K,G)\ra H^{\e 2}(K_{\rm{fl}},\mu)$ induced
by the central extension $1\ra\mu\ra\widetilde{G}\ra G\ra 1$ (see
Remark 3.9(a)). Thus part (i) of the theorem generalizes \cite{Th},
Theorem A, p.458 (from semisimple to arbitrary connected reductive
groups over $K$).

\end{enumerate}
\end{remarks}

Now let $K$ be a global field and let $G$ be a (connected) reductive
algebraic group over $K$. Set
$$
\Sha^{1}\lbe(K,G)=\krn\!\!\left[\e H^{\e 1}(K,G)\ra\displaystyle
\prod_{\text{all $v$}}H^{\e 1}(K_{v},G)\right].
$$
and
$$
\Sha^{1}_{\rm{ab}}\lbe(K,G)=\krn\!\!\left[\e H^{1}_{\rm{ab}}(K_{
\rm{fl}},G)
\ra\displaystyle \prod_{\text{all $v$}}H^{1}_{\rm{ab}} (K_{
v,{\rm{fl}}},G)\right].
$$
\begin{corollary} Let $G$ be a (connected) reductive algebraic group
over a global field $K$. Then the abelianization map
${\rm{ab}}^{1}\colon H^{\e 1}(K,G)\ra H^{1}_{\rm{ab}}(K_{
\rm{fl}},G)$ induces a bijection
$$
\Sha^{1}\lbe(K,G)\simeq\Sha^{1}_{\rm{ab}}\lbe(K,G).
$$
\end{corollary}
\begin{proof} The number field case is due to M.Borovoi \cite{Bor},
Theorem 5.13. The function field case is obtained by applying
Theorem 5.8(i) over $K$ and over the various completions of $K$. See
Remark 5.9(a).
\end{proof}

\begin{remarks}\indent

\begin{enumerate}

\item[(a)] A result similar to the above is known to hold over
the fields of type (ll) mentioned in Remark 5.4(vi). See \cite{BKG},
Theorem 7.1.

\item[(b)] The corollary generalizes the function field case of
\cite{Do5}, Corollary 2.2, from semisimple to arbitrary (connected)
reductive groups.
\end{enumerate}
\end{remarks}

We now recall from \cite{CTGP} that a field $K$ is called {\it of
type (gl)} if it is the function field of a smooth, projective and
connected surface over an algebraically closed field $k$ of
characteristic zero. It is called {\it of type (ll)} if it is the
field of fractions of an excellent, henselian, two-dimensional local
domain $A$ with residue field $k$ as above. If $X$ is a smooth
projective model of $K$ over $k$ (respectively, if $X$ is a regular
and integral two-dimensional scheme equipped with a proper
birational morphism $X\ra\spec A$) and $v$ is a discrete valuation
associated to a point of codimension 1 on $X$, $K_{v}$ will denote
the completion of $K$ at $v$. See \cite{CTGP}, \S 1, for a
description of these fields. Let $G$ be a (connected) reductive
algebraic group over a field of type (gl) or (ll). For any prime $v$
of $K$, let ${\rm{res}}_{\le v}\colon H^{2}(K_{ \rm{fl}},G)\ra
H^{2}(K_{v, {\rm{fl}}},G)$ and ${\rm{res}}_{{\rm{ab}},\le v}\colon
H^{2}_{{\rm{ab}}}(K_{ \rm{fl}},G)\ra H^{2}_{{\rm{ab}}}(K_{v,
{\rm{fl}}},G)$ be the maps of pointed sets induced by $\spec
K_{v}\ra\spec K$. Further, we will write ${\rm{ab}}^{2}_{v}\colon
H^{\e 2}(K_{v, {\rm{fl}}},G)\ra H^{\e 2}_{\rm{ab}} (K_{v,
{\rm{fl}}},G)$ for the second abelianization map associated to
$G\be\times_{K}\! K_{v}$.

\smallskip

The next corollary is analogous to \cite{Bor2}, Theorem
6.8\footnote{Case (ll) of Corollary 5.12 is implicit in \cite{CTGP},
proof of Theorem 5.5.}.

\begin{corollary} Let $K$ be either a field of type (gl), (ll) or a
global function field. Let $G$ be a (connected) reductive algebraic
group over $K$ such that $\Sha^{2}\lbe(K,G^{\tor}\e)=0\!\!$
\footnote{Sufficient conditions for the vanishing of $\Sha^{2}
\lbe(K,G^{\tor}\e)$ can be found in \cite{CTGP}, Theorem
5.5(ii)-(iv), and \cite{Bor2}, Theorem 7.3(ii)-(vi).}. Then a class
$\xi\in H^{\e 2}(K_{\rm{fl}},G)$ is neutral if, and only if,
${\rm{res}}_{v}(\xi)\in H^{\e 2}(K_{v, {\rm{fl}}},G)$ is neutral for
every prime $v$ of $K$.
\end{corollary}
\begin{proof} This is immediate from Theorem 5.8(ii) (applied over
$K$ and over $K_{v}$ for every prime $v$ of $K$).
\end{proof}

The following proposition complements the results of \cite{Bor2} and
concludes this paper.

\begin{proposition} Let $G$ be a (connected) reductive algebraic
group over a number field $K$. Then a class $\xi\in H^{2}_{\rm{ab}}
(K_{\rm{fl}},G)$ is in the image of ${\rm{ab}}^{2}$ if, and only if,
${\rm{res}}_{{\rm{ab}},v}\lbe(\xi)\in H^{\e 2}_{\rm{ab}} (K_{v,
{\rm{fl}}},G)$ is in the image of ${\rm{ab}}^{2}_{v}$ for every real
prime $v$ of $K$.
\end{proposition}
\begin{proof} If $\xi$ is in the image of ${\rm{ab}}^{2}$, then
${\rm{res}}_{{\rm{ab}},v}\lbe(\xi)$ is in the image of
${\rm{ab}}^{2}_{v}$ for every prime $v$ of $K$ by the commutativity
of the diagram
$$
\xymatrix{H^{\le 2}(K_{\rm{fl}},G)\ar[r]^(.5){{\rm{ab}}^{2}}
\ar[d]^{{\rm{res}}_{v}}&
H^{2}_{\rm{ab}}(K_{\rm{fl}},G)
\ar[d]^{{\rm{res}}_{{\rm{ab}},v}}\\
H^{\le 2}(K_{v, {\rm{fl}}},G)\ar[r]^(.5){{\rm{ab}}_{v}^{2}}&
H^{2}_{\rm{ab}}(K_{v, {\rm{fl}}},G)}
$$
(see Remark 4.3). On the other hand, by Proposition 4.1,
$\img\e{\rm{ab}}^{2}$ is a subgroup of $H^{2}_{\rm{ab}}
(K_{\rm{fl}},G)$ and the corresponding quotient group
$\cok\rm{ab}^{\lbe 2}$ injects as a subgroup of $H^{\e
3}\big(K,Z\big(\widetilde{G}\e\big)\big)$. The proposition now
follows from the fact that the canonical map
$$
H^{3}\lbe
\big(K,Z\big(\widetilde{G}\e\big)\big)\ra\prod_{\e\text{$v$ real}}
H^{3}\lbe\big(K_{v},Z\big(\widetilde{G}\e\big) \big)
$$
is an isomorphism \cite{Mi}, Theorem I.4.10(c), p.70.
\end{proof}


\begin{thebibliography}{29}


\bibitem[1]{Bor} Borovoi, M.\emph{ Abelian Galois cohomology of
reductive groups.} Mem. Amer. Math. Soc. {\bf{132}} (1998), no. 626.


\bibitem[2]{Bor2} Borovoi, M.\emph{ Abelianization of the second nonabelian
Galois cohomology.} Duke Math. J. {\bf{72}} (1993), no. 1, 217-239.



\bibitem[3]{BKG} Borovoi, M. and Kunyavski\u{\i}, B.
\emph{ Arithmetical birational invariants of linear algebraic groups
over two-dimensional geometric fields}, with an Appendix by P.Gille,
J. Algebra {\bf{276}} (2004) 292-339.


\bibitem[4]{Br} Breen, L.\emph{ Bitorseurs et cohomologie
non ab\'elienne.} In: The Grothendieck Festschrift, Vol. I, 401–476.
Progr. Math., {\bf{86}}, Birkh\"auser Boston, Boston, MA, 1990.


\bibitem[5]{Br2} Breen, L.\emph{ On the classification of 2-gerbes and
2-stacks.} Asterisque no. {\bf{225}} (1994).



\bibitem[6]{BT} Bruhat, F. and Tits, J.\emph{ Groupes
alg\'ebriques sur un corps local. Chapitre III. Compl\'ements et
applications \`a la cohomologie galoisienne.} J. Fac. Sci. Univ.
Tokyo Sect. IA Math. {\bf{34}}, no. 3 (1987), 671-698.


\bibitem[7]{CTGP}  Colliot-Th\'el\`ene, J.-L., Gille, P. and Parimala,
R.\emph{ Arithmetic of linear algebraic groups over two-dimensional
geometric fields.} Duke Math. J. {\bf{121}}, no.2 (2004), 285-341.


\bibitem[8]{DD} D\`ebes, P. and Douai, J.-C.\emph{ Gerbes and covers.}
Comm. Algebra {\bf{27}}, no.2 (1999), 577-594.


\bibitem[9]{Deb} Debremaeker, R.\emph{ Non abelian cohomology.}
Bull. Soc. Math. Belg. {\bf{29}} (1977), 57-72.

\bibitem[10]{SGA3}  Demazure, M. and Grothendieck, A. (Eds.)
\emph{ Sch\'emas en groupes.} S\'eminaire de G\'eom\'etrie
Alg\'ebrique du Bois Marie 1962-64 (SGA 3). Lecture Notes in Math.
{\bf{151-153}}, Springer, Berlin-Heidelberg-New York, 1972.



\bibitem[11]{Do} Douai, J.-C.\emph{ $2$-Cohomologie Galoisienne des
groupes semi-simples.} Th\`ese de doctorat d'etat, Universit\'e de
Lille, 1976. \'Editions universitaires europ\'eennes, 2010
(available from www.amazon.com).

\bibitem[12]{Do1} Douai, J.-C.\emph{ $2$-Cohomologie Galoisienne des
groupes semi-simples d\'efinis sur les corps locaux.} C. R.
Acad. Sci. Paris S\'er. A {\bf{280}} (1975), 321-323.


\bibitem[13]{Do2} Douai, J.-C.\emph{ Cohomologie Galoisienne des
groupes semi-simples d\'efinis sur les corps globaux.} C. R.
Acad. Sci. Paris S\'er. A {\bf{281}} (1975), 1077-1080.


\bibitem[14]{Do3} Douai, J.-C.\emph{ Cohomologie des sch\'emas en
groupes semi-simples sur les anneaux de Dedekind et sur les courbes
lisses, compl\`etes, irr\'eductibles.} C. R.
Acad. Sci. Paris S\'er. A {\bf{285}} (1977), no. 5, 325-328.



\bibitem[15]{Do4} Douai, J.-C.\emph{ Suites exactes d\'eduites
de la suite spectrale de Leray en cohomologie non ab\'elienne.}
J.Algebra {\bf{79}}, no.1 (1982), 68-77.

\bibitem[16]{Do5} Douai, J.-C.\emph{ Cohomologie des sch\'emas en
groupes sur les courbes d\'efinis sur les corps quasi-finis et loi de
r\'eciprocit\'e.} J.Algebra {\bf{103}}, no.1 (1986), 273-284.


\bibitem[17]{Do6} Douai, J.-C.\emph{ Espaces homog\`enes et arithm\'etique
des sch\'emas en groupes r\'eductifs sur les anneaux de Dedekind.}
J.Th\'eor. Nombres Bordeaux {\bf{7}}, no.1 (1995), 21-26.


\bibitem[18]{Do7} Douai, J.-C.\emph{ Sur la $2$-cohomologie non ab\'elienne
des mod\`eles r\'eguliers des anneaux locaux hens\'eliens.} J.
Th\'eor. Nombres Bordeaux {\bf{21}}, no.1 (2009), 119-129.


\bibitem[19]{Do8} Douai, J.-C.\emph{ $2$-cohomology of schemes of semi-simple
simply connected groups over curves defined over $p$-adic fields.}
Talk delivered at the 27th Journ\'ees Arithm\'etiques, Vilnius, June
2011.

\bibitem[20]{GS} Gille, Ph.\emph{ Serre's conjecture II: a survey.}
In ``Quadratic forms, linear algebraic groups, and cohomology",
Developments in math. {\bf{18}} (2010), 41-56, Springer. (Available
at http://www.math.ens.fr/$\sim$gille/publis/hyderabad.pdf)



\bibitem[21]{Gi} Giraud, J.
\emph{Cohomologie non ab\'elienne.} Die Grundlehren der
mathematischen Wissenschaften, vol. {\bf{179}}. Springer-Verlag,
Berlin-New York, 1971.


\bibitem[22]{GA} Gonz\'alez-Avil\'es, C.\emph{
Abelian class groups of reductive group schemes.} Preliminary
version available at http://arxiv.org/abs/1108.3264



\bibitem[23]{GL} Grandjean, A.R, and Ladra, M..\emph{ $H_{2}(T,G,\partial)$
and central extensions for crossed modules.} Proc. Edinburgh Math.
Soc. {\bf{42}} (1999), 169-177.


\bibitem[24]{SGA4}  Grothendieck, A. and Verdier, J. (Eds.)
\emph{Th\'eorie de Topos et Cohomologie Etale des Sch\'emas.}
S\'eminaire de G\'eom\'etrie Alg\'ebrique du Bois Marie 1963-64 (SGA
$4_{\e\text{III}}$). Lecture Notes in Math. {\bf{305}}, Springer,
Berlin-Heidelberg-New York, 1972.



\bibitem[25]{Ha}  Harder, G. \emph{\"Uber die Galoiskohomologie
halbeinfacher algebraischer Gruppen, III.} J. Reine Angew. Math.
274-275 (1975), 125-138.



\bibitem[26]{J} de Jong, J. \emph{The period-index problem for the
Brauer group of an algebraic surface} Duke Math. J. {\bf{123}} (2004),
71-94.



\bibitem[27]{Mi1} Milne, J.S.:\emph{ \'Etale Cohomology.}
Princeton University Press, Princeton, 1980.


\bibitem[28]{Mi} Milne, J.S.:\emph{ Arithmetic Duality Theorems.}
Persp. in Math., vol. 1. Academic Press Inc., Orlando 1986.



\bibitem[29]{PR} Platonov, V. and Rapinchuk, A.:\emph{ Algebraic
Groups and Number Theory.} Academic Press Inc., 1994.


\bibitem[30]{Sh} Shatz, S.:\emph{ The cohomological dimension of
certain Grothendieck topologies.} Ann. of Math. {\bf{83}}, no.3
(1966), 572-595.


\bibitem[31]{Th} Th\v{a}\'ng, N.:\emph{On Galois cohomology of
semisimple groups over local and global fields of positive
characteristic.} Math.Z. {\bf{259}} (2008), 457-467.




\end{thebibliography}
\end{document}